\documentclass[12pt,thmsa]{article}
\usepackage{amssymb}
\usepackage{amsfonts}
\usepackage{amsmath}
\usepackage[top=3.8cm,bottom=3.8cm,textwidth=16cm,centering,a4paper]{geometry}
\usepackage{verbatim}
\usepackage[numbers,sort&compress]{natbib}

\allowdisplaybreaks[4]
\newtheorem{theorem}{Theorem}[section]

\newtheorem{proposition}[theorem]{Proposition}
\newtheorem{corollary}[theorem]{Corollary}
\newtheorem{remark}[theorem]{Remark}

\newtheorem{lemma}[theorem]{Lemma}

\newenvironment{proof}[1][Proof]{\noindent\textbf{#1.} }{\ \rule{0.5em}{0.5em}}

\begin{document}
\title{Normalized multi-bump solutions for Choquard equation involving sublinear case}
\date{}
\author{He Zhang $^{1}$, Shuai Yao $^{2}$,
	Haibo Chen $^{1}$ \thanks{\small Corresponding author.\newline
    E-mail: hehzhang@163.com (H. Zhang),  shyao@sdut.edu.cn(S. Yao), math\_chb@163.com (H. Chen).}
	\\
	$^1$ {\small School of Mathematics and Statistics, Central South University, Changsha 410083, PR China}
	\\
	$^2${\small School of Mathematics and Statistics, Shandong University of Technology, Zibo 255049, PR China }}

\maketitle
\begin{abstract}
In this paper, we study the existence of normalized multi-bump solutions for the following Choquard equation
\begin{equation*}
-\epsilon^2\Delta u +\lambda u=\epsilon^{-(N-\mu)}\left(\int_{\mathbb{R}^N}\frac{Q(y)|u(y)|^p}{|x-y|^{\mu}}dy\right)Q(x)|u|^{p-2}u, \text{in}\ \mathbb{R}^N,
\end{equation*}
where $N\geq3$, $\mu\in (0,N)$, $\epsilon>0$ is a small parameter and $\lambda\in\mathbb{R}$ appears as a Lagrange multiplier.
By developing a new variational approach, we show that the problem has a family of normalized multi-bump solutions focused on the isolated part of the local maximum of the potential $Q(x)$ for sufficiently small $\epsilon>0$. The asymptotic behavior of the solutions as $\epsilon\rightarrow0$ are also explored. It is worth noting that our results encompass the sublinear case $p<2$, which complements some of the previous works.
\end{abstract}

\section{Introduction}

In this paper, we study the following Choquard equation
\begin{align}
-\epsilon^2\Delta u +V(x) u=\epsilon^{-(N-\mu)}\left(\int_{\mathbb{R}^N}\frac{G(u)}{|x-y|^{\mu}}dy\right)g(u), \text{in}\ \mathbb{R}^N,\label{e1}
\end{align}
where $N\geq3$, $\mu\in (0,N)$, and $\epsilon>0$ is a small parameter.
This problem derives from finding standing wave solutions of the following time-dependdent nonlinear Chouquard equation
\begin{align}
i\epsilon\eta_t=-\epsilon^2\Delta\eta-\epsilon^{-(N-\mu)}\left(\int_{\mathbb{R}^N}\frac{G(\eta)}{|x-y|^{\mu}}dy\right)g(\eta),\ \text{in}\ \mathbb{R}\times\mathbb{R}^N.\label{e2}
\end{align}
To the best of our knowledge, an important feature of problem (\ref{e2}) is the conservation of mass with the $L^2$-norm $\int_{\mathbb{R}^N}|u|^2$ of solutions that is independent of $t\in\mathbb{R}^N$. To find the stationary state, we perform a transformation $\eta(t,x)=e^{it/\epsilon}u(x)$, which still ensures the mass conservation of the $L^2$-norm.

Problem (1.1) has appeared in many fields of mathematical physics, such as physics, laser physics, quantum mechanics, many particle system, etc. When $N=3$, $V(x)=1$, $p=2$, $\mu=N-2$, it is often referred to as the Choquard-Pekar equation and has given rise to the one-component plasma theory that describes electrons trapped in their own holes \cite{L} and the model of stationary polaritons in quantum physics \cite{Pekar}. It was introduced in the 1990s by Jones \cite{J} and Penrose \cite{P} as a model for self-gravitating matter and is known as the Schr\"{o}dinger-Newton equation.
In quantum physics models, the parameter $\epsilon$ is usually the quantised Planck's constant, which is usually very small. In quantum physics, it is expected that in the semi-classical limit as $\epsilon\rightarrow0$, the dynamics should be governed by $V(x)$.

Mathematically, as $\epsilon\rightarrow0$, the existence and asymptotic behavior of the solution to the equation (\ref{e1}) is called a semiclassical problem, which is used to describe the transition between quantum mechanics and Newtonian mechanics. Replacing the Riesz potential with a Dirac distribution, the Choquard equation (\ref{e2}) reduces to the following equation
\begin{align}
-\epsilon^2\Delta u +V(x) u=g(u),\ \text{in}\ \mathbb{R}^N.\label{e3}
\end{align}
When the nonlinear term $g(u)$ satisfies various conditions, the existence of solutions to problem (\ref{e3}) has been extensively studied. The results for multi-bump solutions without $L^2$-constant can be referred to \cite{BW, BJ, G} and the references therein. As for the constrained problem,  Ackermans and Weth \cite{AW} proved that there are infinitely many geometric distance normalized multi-bump solutions as $\epsilon_n\rightarrow0$, by assuming a periodic potential with a non-degenerate critical point. Using Lyapunov-Schmidit reduction, Pellacci et al. \cite{PPVV} obtained a bump solution concentrates as $\epsilon\rightarrow 0$.
Zhang et al. \cite{ZZ} studied the existence of normalized multi-bump solutions to the following nonlinear Schr\"{o}dinger equation
\begin{align}
\begin{cases}
-\Delta u +\lambda u=Q(\epsilon x)|u|^{2\xi}u,\ &\text{in}\ \mathbb{R}^N,\\
|u|^2_2=1,\quad u(x)\rightarrow0\ &\text{as}\ |x|\rightarrow\infty,
\end{cases}\label{e4}
\end{align}
where $\epsilon>0$ is a small parameter,  $\xi\in(0, 2/N)\cup(2/N,2^{\ast}/N)$ and $Q$ satisfies
\begin{itemize}
\item[$(Q)$] $Q\in C(\mathbb{R}^N, (0,+\infty))\cap L^{\infty}(\mathbb{R}^N)$ and there are  $ k\geq2$ mutually disjoint bounded domains $\Omega_i\subset\mathbb{R}^N$, $i=1,2,\cdots, k$ such that $
   M_i:=\max\limits_{x\in\Omega_i} Q(x)>\max\limits_{x\in\partial\Omega_i} Q(x)$.
\end{itemize}
Note that $\mathcal{A}_i:=\{ x\in \Omega_i:\ Q(x)=M_i\}$ is nonempty and compact.
By using variational methods and a penalization technique \cite{BW}, they obtained that equation (\ref{e4}) admits a family of  solutions concentracting at local maximum points of $Q$ when $\epsilon\rightarrow 0$. We find that these semi-classical solutions have a higher topological type, since they are obtained from minimax characterization of higher-dimensional symmetric connected structures.

When $\epsilon=1$, the existence of normalized solutions for problem (\ref{e1}) attracts much attentions of researchers, see
\cite{BLL,Luo,PPVV,YCRS} and the references therein.
The existence of semiclassical states for non-local problems (\ref{e1}) was first mentioned in [\cite{AM}, P. 29]. In $N=3$, $\mu=N-2$ and $F(s)=s^2$,
using  Lyapunov-Schmidt reduction method, Wei and Winter \cite{WW} obtained a family of solutions to the Schr\"{o}dinger-Newton equations that are concentrated  around a non-degenerate critical point of the potential $V(x)$ as $\epsilon\rightarrow0$, where $V>0$ and $\liminf_{|x|\rightarrow\infty}V(x)|x|^r>0$ for some $r\in [0,1)$. The Lyapunov-Schmidt reduction method depends on the uniqueness of positive ground state solutions to the limit problem:
\begin{align*}
-\Delta u+u=\left(\int_{\mathbb{R}^N}\frac{|u(y)|^2}{|x-y|}dy\right),\ \text{in}\ \mathbb{R}^3,
\end{align*}
which was established by Lieb in \cite{L}. Subsequently, Moroz and Van Schaftingen \cite{MV1} studied the following Choquard equation
\begin{align}
-\epsilon^2\Delta u +V(x) u=\epsilon^{-(N-\mu)}\left(\int_{\mathbb{R}^N}\frac{|u(y)|^p}{|x-y|^{\mu}}dy\right)|u|^{p-2}u, \text{in}\ \mathbb{R}^N.\label{e5}
\end{align}
By using the variational methods and a novel nonlocal penalization technique unlike \cite{BJ}, they proved the existence of a single-peak solution for (\ref{e5}) concentrating at a local minimum of $V$ when $p\in [2,\frac{2N-\mu}{N-2})$.  Using a penalization technique \cite{BJ}, Yang et al. \cite{YZZ} extended the results in \cite{MV1} and proved the existence of multi-peak solutions, where each peak concentrates at different local minimum of $V(x)$ as $\epsilon\rightarrow0$ when $\lim\limits_{t\rightarrow\infty}\frac{g(t)}{t^{\frac{N-\mu+2}{N-2}}}=0$ and
\begin{align}
\lim\limits_{t\rightarrow0}\frac{g(t)}{t}=0.\label{e6}
\end{align}
Liu, Ma and Xia \cite{LMX} constructed a family of localized bound states of higher topological type that concentrate around the local minimum points of the potential $V$ as $\epsilon\rightarrow0$ when $p\in [2,(2N-\mu)/(N-2))$.
For the influence of the potential well topology on the existence of multi-bump solutions, we can refer to \cite{DGY,HRZ, SL,  ZZ1} and the references therein.

We remark that most of the above results require the exponent satisfying the condition of $p\geq2$ or (\ref{e6}), which is important in their argument because it allows them to use the linearized problem at infinity. For the existence and asymptotics of the Choquard equation, we refer to \cite{MV, MV1} and the references therein.
As for the sublinear case $p\in \left(\frac{2N-\mu}{N},2\right)$, by using the deformation flow method, Cingolani and Tanaka \cite{CT} obtained the existence and multiplicity of positive single peak solutions for the Choquard equation via the cup-length of a critical set $Crit_{V_0}=\{x\in\Omega: V(x)=V_0\}$, where $\Omega\in \mathbb{R}^N$ is a boundeded set such that $V_0\equiv\inf\limits_{x\in\Omega}V(x)<\inf\limits_{x\in\partial\Omega}V(x)$.
 Subsequently, they generalised the results in \cite{CT} and proved the existence of a family of solutions in \cite{CT2} concentrated on local maxima or saddle points of $V(x)$.
However, when $\epsilon$ is small enough, we find that the existence of the normalization solution of equation (\ref{e1}) seems to have no result. Hence,
in the present paper, we study the construction of normalized multi-bump solutions for (\ref{e1}) without nondegeneracy assumptions on the potential function. To illustrate our problem, let $V(x)=1$, $p\in \left(\frac{2N-\mu}{N}, \frac{2N-\mu+2}{N}\right)\cup \left(\frac{2N-\mu+2}{N}, \frac{2N-\mu}{N-2}\right)$, $g(x)=Q(x)|u|^{p-2}u$ and by scaling,  Eq. (\ref{e1}) becomes the following Choquard equation:
\begin{align}
\begin{cases}
-\Delta v +\lambda v=\left(\int_{\mathbb{R}^N}\frac{Q(\epsilon y)|v(y)|^p}{|x-y|^{\mu}}dy\right)Q(\epsilon x)|v|^{p-2}v, &\text{in}\ \mathbb{R}^N,\\
|v|^2_2=1,\quad v(x)\rightarrow0\ &\text{as}\ |x|\rightarrow\infty,
\end{cases}\label{e7}
\end{align}

Inspired by \cite{ZZ}, we can construct the multi-peak solutions of (\ref{e7}) with the positive ground state vector solutions of the following limiting Choquard system:
\begin{align}\label{e8}
\begin{cases}
-\Delta u+\lambda u=M_i^2\left(\int_{\mathbb{R}^N}\frac{|u(y)|^p}{|x-y|^{\mu}}\right)|u(x)|^{p-2}u\quad \text{in}\ \mathbb{R}^N,\\
|u|^2_2=\frac{M_i^{\frac{-2}{p-1}}}{\sum^{ k}_{i=1}M_i^{\frac{-2}{p-1}}},\ u>0,\ \lim\limits_{|x|\rightarrow\infty}u(x)=0,
\end{cases}
\end{align}
where $M_i$ is given as $(Q)$, $i=1,\cdots,k$.  Let $(a,b)\in (\mathbb{R}^+,\mathbb{R}^+)$ replace $\left(M_i, M_i^{\frac{-2}{p-1}}/\sum^{ k}_{i=1}M_i^{\frac{-2}{p-1}}\right)$. Then system (\ref{e8}) becomes the following equation:
\begin{align}
\begin{cases}
-\Delta v+\lambda v=b\left(\int_{\mathbb{R}^N}\frac{|v(y)|^p}{|x-y|^{\mu}}\right)|v(x)|^{p-2}v\quad \text{in}\ \mathbb{R}^N,\\
|v|^2_2=a,
\end{cases}\label{e9}
\end{align}
where $a,b>0$ is a constant.
Define the energy functional $E_b:H^1(\mathbb{R}^N)\rightarrow\mathbb{R}$ corresponding to equation (\ref{e9}) by
\begin{align*}
E_b(u)=\frac{1}{2}\int_{\mathbb{R}^N}|\nabla u|^2dx-\frac{b}{2p}\int_{\mathbb{R}^N}\int_{\mathbb{R}^N}\frac{|u(x)|^p|u(y)|^p}{|x-y|^{\mu}}.
\end{align*}
Clearly, $E_b\in C^1(H^1(\mathbb{R}^N),\mathbb{R}^N)$ and $(\lambda, u)$ solves (\ref{e9}) if and only if $u$ is a critical point of $E_b$ restricted on the constraint
 $ S_{b,a}=\{u\in H^1(\mathbb{R}^N)|\ |u|^2_2=a, \mathcal{P}_b(u)=0\}$ ,
where
\begin{align*}
\mathcal{P}_b(u)=\int_{\mathbb{R}^N}|\nabla u|^2dx-\frac{b(Np-\mu)}{2p}\int_{\mathbb{R}^N}\int_{\mathbb{R}^N}\frac{|u(x)|^p|u(y)|^p}{|x-y|^{\mu}}.
\end{align*}
Define
\begin{equation*}
E_b(a)=\inf\limits_{u\in S_{b,a}}E_b(u).
\end{equation*}
Then we have the following results which have been presented in \cite{Y, CT1, Luo}.
\begin{theorem}\label{T1.1}
Let $a,b>0$, $p\in \left(\frac{2N-\mu}{N},\frac{2N-\mu}{N-2}\right)\setminus\{\frac{2N-\mu+2}{N}\}$. Then $E_b(a)$ is attained. That is, there exists $(\lambda, u_a)\in (\mathbb{R},S_{b,a})$ such that $E_b(u_a)=E_b(a)$.
\end{theorem}

\begin{remark}\label{T1.2}
$(i)$ In the $L^2$-subcritical case $p\in \left(\frac{2N-\mu}{N},\frac{2N-\mu+2}{N}\right)$,  Ye \cite{Y} or Silvia Cingolani and Kazunaga Tanaka \cite{CT1} have shown that $E_b(a)$ is attained.

$(ii)$ In the $L^2$-supercritical  case $p\in \left(\frac{2N-\mu+2}{N},\frac{2N-\mu}{N-2}\right)$, we can extend the results of \cite{Luo} from the case of $p=2$ to $p\in \left(\frac{2N-\mu+2}{N},\frac{2N-\mu}{N-2}\right)$ by applying Schwarz spherical rearrangement. Hence, We're not going to prove it here.
\end{remark}

According to the results of theorem \ref{T1.1}, we discuss that existence of positive ground state solution of (\ref{e8}). The results are as follows:

\begin{theorem}\label{T1.3}
Let $N\geq 3$, $p\in\left(\frac{2N-\mu}{N},\frac{2N-\mu}{N-2}\right)\setminus\{\frac{2N-\mu+2}{N}\}$ and $M_i>0( i=1,\cdots,k)$ given in $(Q)$.
Assume that $(\lambda, u_{M_i})$ is the ground state solution of problem (\ref{e9}) by replacing $(a,b)$ with $\left(M_i, \frac{M_i^{\frac{-2}{p-1}}}{\sum^{ k}_{i=1}M_i^{\frac{-2}{p-1}}}\right)$.
Then  $(\lambda, u_{M_1}, \cdots, u_{M_k})$ is a normalized ground state vector solution of system (\ref{e8}). Furthermore, there holds:
\begin{align*}
\sum_{i=1}^{k}E_{M_i}(a_i)<\sum_{i=1}^{k}E_{M_i}\left(\frac{M_i^{\frac{-2}{p-1}}}{\sum^{ k}_{i=1}M_i^{\frac{-2}{p-1}}}\right),\ \text{if}\ p\in \left(\frac{2N-\mu}{N}, \frac{2N-\mu+2}{N}\right),\\
\sum_{i=1}^{k}E_{M_i}(a_i)>\sum_{i=1}^{k}E_{M_i}\left(\frac{M_i^{\frac{-2}{p-1}}}{\sum^{ k}_{i=1}M_i^{\frac{-2}{p-1}}}\right),\ \text{if}\ p\in \left(\frac{2N-\mu+2}{N}, \frac{2N-\mu}{N-2}\right),
\end{align*}
where $a_i\in (0,1)\setminus\left\{\frac{M_i^{\frac{-2}{p-1}}}{\sum^{ k}_{i=1}M_i^{\frac{-2}{p-1}}}\right\}, i=1,\cdots,k$ and $\sum_{i=1}^{k}a_i=1$.
\end{theorem}
 \begin{remark}\label{T1.4}
 If $N=3, p=2$ and $ \mu=N-2$, then the solution of systems (\ref{e8}) and (\ref{e9}) is unique.
 \end{remark}

From the results of Theorems \ref{T1.1} and \ref{T1.3}, we construct a family of solutions of problem (\ref{e7}).

\begin{theorem}\label{T1.5}
Suppose $(Q)$ holds and $p\in \left(\frac{2N-\mu}{N},\frac{2N-\mu}{N-2}\right)\setminus\{\frac{2N-\mu+2}{N}\}$. Then $\epsilon_0>0$ exists such that for each $\epsilon\in (0,\epsilon_0)$, problem $(\ref{e7})$ has a solution $(\lambda_{\epsilon}, v_{\epsilon})\in \mathbb{R}\times
 H^1(\mathbb{R}^N)$. Furthermore, there is $ k$ local maximum points $x_{\epsilon}^i$ ($i=1, \cdots,  k$, $k\geq2$) of $v_{\epsilon}>0$ such that $\lim\limits_{\epsilon\rightarrow0}\max\limits_{1\leq i\leq k} dist(x_{\epsilon}^i,\mathcal{A}_i)=0$
and $\lambda_{\epsilon}\rightarrow \lambda_0$, $v_{\epsilon}(\epsilon x+x_{\epsilon}^i)
 \rightarrow u_i$ as $\epsilon\rightarrow0$,
 where $(\lambda_0, u_i)$ is the normalized ground state vector solution to (\ref{e8}).
\end{theorem}

\begin{remark}\label{T1.6}
$(i)$  In Theorem \ref{T1.5}, we also show that the solutions of problem (\ref{e7}) is concentrated on the normalized ground state solutions of the limit problem (\ref{e8}), which does not require the solution of problem (\ref{e8}) to be unique.

$(ii)$ As shown by theorem \ref{T1.5}, our results contain the case of $p<2$, which extends the results in \cite{MV1,YZZ, LMX}.
\end{remark}

To prove Theorem \ref{T1.5}, we develop a new variational argument which relies and extends some works by  \cite{CT} and \cite{ZZ}.
In \cite{ZZ}, Zhang et al. use the penalization technique and the uniqueness of the solution to the following problem
\begin{align*}
\begin{cases}
-\Delta u_i+\lambda u_i=M_i |u_i|^{2\xi}u_i \ \text{in}\ \mathbb{R}^N,\\
u_i(x)>0,\ \lim\limits_{x\rightarrow\infty}u_i(x)=0,\ i=1,\cdots, k,\\
\sum\limits_{i=1}^{ k}|u_i|^2_2=1,
\end{cases}
\end{align*}
which is very important to the concentration of compactness results. However, due to the existence of non-local term $\left(\int_{\mathbb{R}^N}\frac{|u(y)|^p}{|x-y|^{\mu}}\right)|u(x)|^{p-2}u$, we face the following two difficulties: $(i)$ the solution of system (\ref{e8}) is not necessarily unique. $(ii)$ Penalization technique depends on the decay property of linear equations outside a small ball. For this purpose, the exponential decay of the solution is crucial, which is valid when $p\geq 2$. However, it is invalid for the condition $p<2$ since little is known about the asympototic behaviour for the bound state solutions of the Choquard equations.
Thus, to overcome these difficulties, we need to adopt other methods and techniques.

First of all,  we obtain the properties of the positive ground state solutions of system (\ref{e8}). Then according to these properties, we deal with the difficulties arising from the non-uniqueness of the solutions of system (\ref{e8}).
 We must point out that  we can obtain the existence of normalized multi-bump solutions to problem (\ref{e7}) for $p\geq 2$ by using a penalization technique as in \cite{ZZ}. However, it can not deal with the case of $p<2$. Hence, in this paper, to obtain the results for $p\in\left(\frac{2N-\mu}{N}, \frac{2N-\mu}{N-2}\right)\setminus\{\frac{2N-\mu+2}{N}\}$ which contains the case of $p<2$ , we take a new approach which is inspired by \cite{CT} without using Penalization technique and Lyapunov-Schmidt reduction method.
Similar to \cite{CT}, we also find an operator $J_{\epsilon}$, which has similar properties to the tail minimizing operator in \cite{BT}. Using the tail minimizing flow $J_{\epsilon}$ together with the standard deformation argument, we can directly deal with energy function $E_{\epsilon}$ in $D^{\epsilon}_d$. $E_{\epsilon}$, $D^{\epsilon}_d$ are defined in Section 3.
In the present paper, we do not require the monotonicity condition or Ambrosetti-Rabinowtiz condition \cite{AR}.

The paper is organized as follows. In Section 2, we discuss some properties about the limit system (\ref{e8}) and give the proof of Theorem \ref{T1.3}. In Section 3, we first give a concentration compactness type result. Following this, we develop a deformation argument for a tail of $u\in D^{\epsilon}_d$ to construct a map $J_{\epsilon}$. Eventually, we present the proof of Theorem \ref{T1.5}.

\section{The properties of ground state to limit problem}
~~~~In this section we give some preliminary results and study the existence and properties of normalized ground state solutions to system (\ref{e8}).

First, we give the following Hardy-Little-Sobolev inequality.
\begin{proposition}\label{L2-1} (Hardy-Little-Sobolev inequality \cite{LM})
Let $q,r>1$ and $\mu\in (0,N)$ with $1/q+1/r=\frac{2N-\mu}{N}$. Then there exists a constant $C=(N,\mu,q,r)>0$ such that
\begin{align*}
\left|\int_{\mathbb{R}^N}\int_{\mathbb{R}^N}\frac{f(x)g(y)}{|x-y|^{\mu}}dx\right|\leq C|f|_q|g|_r
\end{align*}
for all $f\in L^q(\mathbb{R}^N)$ and $g\in L^r(\mathbb{R}^N)$.
\end{proposition}

We also give the following inequality used in the later section.

\begin{proposition}\label{L2-2} (\cite{CT})
Let $q,r>1$ and $\mu\in (0,N)$ with $1/q+1/r<\frac{2N-\mu}{N}$, then there exists a constant $D_R>0$ depending on $R>0$ such that $\lim\limits_{R\rightarrow\infty}D_R=0$
and
\begin{align*}
\left|\int_{\mathbb{R}^N}\int_{\mathbb{R}^N}\frac{f(x)g(y)}{|x-y|^{\mu}}dx\right|\leq D_R|f|_q|g|_r
\end{align*}
for all $f\in L^q(\mathbb{R}^N)$ and $g\in L^r(\mathbb{R}^N)$ with $\rm{dist}(\rm{supp}\ f, \rm{supp}\ g)\geq R$.
\end{proposition}

Observe that when $k=1$, system (\ref{e8}) is equivalent to problem (\ref{e9}). Hence, we first consider the properties of problem (\ref{e9}).
We define
\begin{align*}
\overline{S}_{b,a}=\{ u\in  S_{b,a}| E_b(u)=E_b(a)\}.
\end{align*}
From Theorem \ref{T1.1}, we know that $E_b(a)$ is attained and hence $\overline{S}_{b,a}$ is non-empty.
Then we discuss the relationship between the normalized ground state solution of problem (\ref{e9}) and the mass $a$ with $L^2$-norm, the coefficient  $b$ of the nonlinearity.
\begin{lemma}\label{L2-3}
Let the constants $s, a,b>0$, $p\in\left(\frac{2N-\mu}{N},\frac{2N-\mu}{N-2}\right)\setminus\left\{\frac{2N-\mu+2}{N}\right\}$ . Then we have the following statements:

$(i)$ $v\in  \overline{S}_{b,a}$ if and only if $u(\cdot)=s^{\frac{2+N-\mu}{2(-Np+2N+2-\mu)}}v(s^{\frac{p-1}{-Np+2N+2-\mu}}\cdot)\in\overline{ S}_{b,sa}$
and $ S_{b,d}\neq\emptyset$. Furthermore, there holds
\begin{equation}
 E_b(sa)=s^{\frac{-Np+2N-\mu+2p}{-Np+2N+2-\mu}} E_b(a).\label{e10}
\end{equation}

$(ii)$ $v\in  \overline{S}_{b,a}$ if and only if $u(\cdot)=s^{\frac{N}{2(-Np+2N+2-\mu)}}v(s^{\frac{1}{-Np+2N+2-\mu}}\cdot)\in \overline{S}_{sb,a}$. Moreover, there holds
\begin{align}
    E_{sb}(a)=s^{\frac{-Np+2N-\mu+2p}{-Np+2N+2-\mu}} E_b(a).\label{e11}
\end{align}

$(iii)$ There holds:
\begin{align*}
E_{b}(a)<0, \text{if}\ p\in \left(\frac{2N-\mu}{N},\frac{2N-\mu+2}{N}\right),\\
E_{b}(a)>0, \text{if}\ p\in \left(\frac{2N-\mu+2}{N},\frac{2N-\mu}{N-2}\right).
\end{align*}

$(iv)$ $E_b(a)$ is continuous and strictly decremented for $a, b>0$.
\end{lemma}

\begin{proof}
$(i)$
By calculating, we have
\begin{align*}
|u|^2_2=s|v|^2_2,\quad |\nabla u|^2_2=s^{\frac{-Np+2N-\mu+2p}{-Np+2N+2-\mu}}|\nabla v|^2_2,
 \int_{\mathbb{R}^N}\int_{\mathbb{R}^N}\frac{|u|^p|u|^p}{|x-y|^{\mu}}=s^{\frac{-Np+2N-\mu+2p}{-Np+2N+2-\mu}}
 \int_{\mathbb{R}^N}\int_{\mathbb{R}^N}\frac{|v|^p|v|^p}{|x-y|^{\mu}},
\end{align*}
for $u,v\in H^1(\mathbb{R}^N)$ satisfying $u(\cdot)=s^{\frac{2+N-\mu}{2(-Np+2N+2-\mu)}}v(s^{\frac{p-1}{-Np+2N+2-\mu}}\cdot)\in \overline{S}_{b,sa}$, which implies that
$v\in  \overline{S}_{b,a}$ if and only if $u\in \overline{S}_{b,sa}$ and thus (\ref{e10}) holds.

$(ii)$ A direct calculation gives that
\begin{align*}
|u|^2_2=|v|^2_2,\quad |\nabla u|^2_2=s^{\frac{2}{-Np+2N+2-\mu}}|\nabla v|^2_2,
 \int_{\mathbb{R}^N}\int_{\mathbb{R}^N}\frac{|u|^p|u|^p}{|x-y|^{\mu}}=s^{\frac{Np-\mu}{-Np+2N+2-\mu}}
 \int_{\mathbb{R}^N}\int_{\mathbb{R}^N}\frac{|v|^p|v|^p}{|x-y|^{\mu}},
\end{align*}
for $u,v\in H^1(\mathbb{R}^N)$ satisfying $u(\cdot)=s^{\frac{N}{2(-Np+2N+2-\mu)}}v(s^{\frac{1}{-Np+2N+2-\mu}}\cdot)\in \overline{S}_{sb,a}$ and thus (\ref{e11}) holds.

$(iii)$  From \cite{Y, CT1, Luo}, we complete of the proof of $(iii)$.

$(iv)$
According to (\ref{e10})-(\ref{e11}) and the results of $(iii)$, we easily deduce that $(iv)$ holds.
\end{proof}

Next, we discuss the existence and properties of system (\ref{e8}). For convince, let $b_i=M_i$ and $s_i=\frac{M_i^{\frac{-2}{p-1}}}{\sum^{ k}_{i=1}M_i^{\frac{-2}{p-1}}}$ where $M_i$ is given as $(Q)$.
 Then we study the existence and properties of positive solutions for the following system:
\begin{align}
\begin{cases}
-\Delta v_i+\lambda v_i=b_i\left(\int_{\mathbb{R}^N}\frac{|v_i(y)|^p}{|x-y|^{\mu}}\right)|v_i(x)|^{p-2}v_i\quad \text{in}\ \mathbb{R}^N,\\
v_i(x)>0,\ \lim\limits_{x\rightarrow \infty}v_i(x)=0,\ i=1,2,\dots, k,\\
|v_i|^2_2=s_i,\ \sum_{i=1}^{ k} s_i =1.
\end{cases}\label{e12}
\end{align}
Set $
\mathcal{S}_{k}:=\left\{s=(s_1,\cdots,s_{ k})\in (0,1)^{ k}|\sum_{i=1}^{ k}s_i=1\right\}$
and  $\sigma_s:=\sum^{ k}_{i=1}E_{b_i}(s_i)$ for $s=(s_1,\cdots,s_{ k})\in \mathcal{S}_{k}$.
From Lemma \ref{L2-3}, we know that $E_{b_i}(s_i)$ is achieved for each $i=1,\cdots, k$.
Next, we study the properties of $\sigma_s$ according to the following lemma, which plays a significant role in later chapters.

Denote
\begin{align}\label{e13}
s^0=(s_1^0,\cdots, s_k^0)=\left(\frac{b_1^{\frac{-1}{p-1}}}{\sum^{ k}_{i=1}(b_i)^{\frac{-1}{p-1}}},\cdots,
\frac{b_{ k}^{\frac{-1}{p-1}}}{\sum^{ k}_{i=1}b_{ k}^{\frac{-1}{p-1}}}\right).
\end{align}
Then we discuss the relationship between $\sigma_s$ ($s\in\mathcal{S}_{k}\setminus\ {s^0}$) and $\sigma_{s^0}$ which plays an important role in the proof of compactness.

\begin{lemma}\label{L2-4}
Let $s^0$ is given as (\ref{e13}). Then we have:

$(i)$ $\sigma_{s^0}=E_{1}(1)\left(\sum_{i=1}^{ k}(b_i)^{-\frac{1}{p-1}}\right)^{\frac{-2p+2}{-Np+2N+2-\mu}}$ and $\sigma_{s^0}$ is continuous and strictly decreasing  with respect to $b_i$, $i=1,\cdots, k$.

$(ii)$ For each $s\in \mathcal{S}_{k}\setminus\ {s^0}$,
there holds $\sigma_s<\sigma_{s^0}$ for any $p\in\left(\frac{2N-\mu}{N}, \frac{2N-\mu+2}{N}\right)$ and
$\sigma_s>\sigma_{s^0}$ for any $p\in\left(\frac{2N-\mu+2}{N},\frac{2N-\mu}{N-2}\right)$.
\end{lemma}
\begin{proof}
$(i)$ From Lemma \ref{L2-3} $(i)-(ii)$, we can easily get this conclusion.

$(ii)$
Let $s=(s_1,\cdots,s_{ k})\in \mathcal{S}_{k}\setminus\ {s^0}$.
It follows from Lemma \ref{L2-3} $(i)-(ii)$ that
\begin{align*}
\sigma_s=\sum_{i=1}^{ k}E_{b_i}(s_i)
=&E_{1}(1)\sum_{i=1}^{ k}\left(s_i\right)^{\frac{-Np+2N-\mu+2p}{-Np+2N+2-\mu}}\left(b_i\right)^{\frac{2}{-Np+2N+2-\mu}}\\
=&E_{1}(1)\left(\sum_{i=1}^{ k}\left(b_i\right)^{-\frac{1}{p-1}}\right)^{\frac{-2p+2}{-Np+2N+2-\mu}}
\sum_{i=1}^{ k}\left(s_i\right)^{\frac{-Np+2N-\mu+2p}{-Np+2N+2-\mu}}(s^0_i)^{\frac{-2p+2}{-Np+2N+2-\mu}}\\
=&\sigma_{s^0}
\sum_{i=1}^{ k}\left(s_i\right)^{1+\frac{-(2p-2)}{-Np+2N+2-\mu}}(s^0_i)^{\frac{-(2p-2)}{-Np+2N+2-\mu}}.
\end{align*}
For convince, we define
\begin{align*}
A_{p}:=
\begin{cases}
\frac{2p-2}{-Np+2N+2-\mu},\quad &\text{if}\ p\in\left(\frac{2N-\mu}{N}, \frac{2N-\mu+2}{N}\right),\\
-1-\frac{2p-2}{-Np+2N+2-\mu},\quad &\text{if}\ p\in\left(\frac{2N-\mu+2}{N}, \frac{2N-\mu}{N-2}\right).
\end{cases}
\end{align*}
Clearly, $A_{p}>0$.
By the H\"{o}der inequality, we compute that for $p\in\left(\frac{2N-\mu}{N}, \frac{2N-\mu+2}{N}\right)$,
\begin{align*}
\left[\sum_{i=1}^{ k}(s_i)^{1+A_p}(s^0_i)^{-A_p}\right]^
{\frac{1}{1+A_p}}
=&\left[\sum_{i=1}^{ k}(s_i)^{1+A_p}(s^0_i)^{-A_p}\right]^{\frac{1}{1+A_p}}\left(\sum_{i=1}^{ k}s_i^0\right)^{A_p/(1+A_p)}\\
\geq&\sum_{i=1}^{ k}s_i (s_i^0)^{-A_p/(1+A_p)}(s_i^0)^{A_p/(1+A_p)}=1
\end{align*}
and  equality holds if and only if $s_i=s_i^0$, $i=1,2,\cdots, k$.
As for $p\in\left(\frac{2N-\mu+2}{N}, \frac{2N-\mu}{N-2}\right)$, following a similar method, we also obtain $
\left[\sum_{i=1}^{ k}(s_i^0)^{1+A_p}(s_i)^{-A_p}\right]^{\frac{1}{1+A_p}}\geq 1$
and equality holds if and only if $s_i=s_i^0$, $i=1,2,\cdots, k$. Combining with lemma \ref{L2-3} $(iii)$, we complete the proof.
\end{proof}

\begin{corollary}\label{L2-5}
Let $s^0$ is given as (\ref{e13}). Then there holds that
$(\lambda_0, (u_1,\cdots,u_{ k}))$ is the solution of system (\ref{e12}) where
$u_i\in \overline{S}_{b_i, s^0_i}$ and
\begin{align}\label{e14}
\lambda_0=\frac{2(-Np+2N-\mu+2p)}{Np-N-2-(N-\mu)}\left(\sum_{i=1}^{ k}(b_i)^{-\frac{1}{p-1}}\right)^{\frac{-2p+2}{-Np+2N+2-\mu}}E_{1}(1).
\end{align}
\end{corollary}

The solution of  system (\ref{e12}) is not necessarily unique.  Therefore, the following lemma is needed.

\begin{lemma}\label{L2-6}
Let $(\lambda, (u_1,\cdots,u_{ k}))$ is the solution of system $(\ref{e12})$. Then one has $\lambda\leq \lambda_0$ and $\sum_{i=1}^{ k}E_{b_i}(u_i)\geq \sigma_{s^0}$. In particular, equalities hold if and only if
 $\lambda= \lambda_0$ and $u_i\in \overline{S}_{b_i, s^0_i}$, $i=1,\cdots,k$.
 where $s_i^0$, $\lambda_0$ are given as (\ref{e13}) and (\ref{e14}) respectively.
\end{lemma}
\begin{proof}
Let $|u_i|^2_2=s_i$, $\sum\limits_{i=1}^{ k}s_i=1$ for each $i=1,\cdots, k$. For convince, we denote $
A=\frac{2(-Np+2N-\mu+2p)}{-Np+2N+2-\mu}$.
If $p\in\left(\frac{2N-\mu}{N},\frac{2N-\mu+2}{N}\right)$, then $A>0$ and
a direct calculation gives that
\begin{align*}
\lambda=-A\frac{E_{b_i}(u_i)}{s_i}
\leq-A\frac{E_{b_i}(s_i)}{s_i}=-AE_{1}(1)
(s_i)^{\frac{2p-2}{-Np+2N+2-\mu}}(b_i)^{\frac{-2}{-Np+2N+2-\mu}},
\end{align*}
where $i=1,\cdots, k$.
If $s_i\equiv s_i^0, i=1,\cdots, k$, the proof is completed. If $s_i\not\equiv s_i^0, i=1,\cdots, k$, there exists $i$ such that $s_i<s_i^0$. Clearly, there holds
\begin{align*}
\min\left\{(s_i)^{\frac{2p-2}{-Np+2N+2-\mu}}(b_i)^{\frac{-2}{-Np+2N+2-\mu}},\ i=1,\cdots, k\right\}<(s^0_i)^{\frac{2p-2}{-Np+2N+2-\mu}}(b_i)^{\frac{-2}{-Np+2N+2-\mu}}
 \end{align*}
which implies $\lambda<\lambda_0$ and $\sum\limits_{i=1}^{ k}E_{b_i}(u_i)>
\sum\limits_{i=1}^{ k}E_{b_i}(s_i^0)=\sigma_{s^0}$.
If $p\in \left(\frac{2N-\mu+2}{N},\frac{2N-\mu}{N-2}\right)$, it follows from Lemma \ref{L2-4} $(ii)$ that the conclusion is true.
\end{proof}

Denote $\mathcal{S}=\overline{S}_{b_1, s_1^0}\times\cdots\times\overline{S}_{b_{ k}, s_{ k}^0}$ be the set of normalized ground state  solutions of system (\ref{e12}), where  $s_1^0,\cdots, s_k^0$ is given as (\ref{e13}). For each $\mathbf{U}=(u_1,\cdots,u_{ k})\in \mathcal{S}$, without loss of generality, we suppose $u_i(0)=\max\limits_{x\in\mathbb{R}^N}u_i(x)$. Then we have the following conclusion.

\begin{lemma}\label{L2-7}

Let $\mathcal{S}=\overline{S}_{b_1, s_1^0}\times\cdots\times\overline{S}_{b_{ k}, s_{ k}^0}$ where $s_i^0$ is given as (\ref{e13}). Then we have that
$\mathcal{S}$ is a bounded set and compact in $(H^1(\mathbb{R}^N))^{ k}$.
\end{lemma}

\begin{proof}
Firstly, we claim that $\mathcal{S}$ is a bounded set in $(H^1(\mathbb{R}^N))^{ k}$.
 Let $(u_1,\cdots,u_{ k})\in \mathcal{S}$, then we prove that $u_i$ is bounded in $H^1(\mathbb{R}^N)$ for each $i=1,\cdots, k$. Clearly, $u_i\in \overline{S}_{b, s_i^0}$. This indicates
\begin{align*}
&E_{b_i}(s_i^0)=\frac{1}{2}\int_{\mathbb{R}^N}|\nabla u_i|^2-\frac{b_i}{2p}\int_{\mathbb{R}^N}\int_{\mathbb{R}^N}\frac{|u_i(x)|^p|u_i(y)|^p}{|x-y|^{\mu}},\\
&\mathcal{P}_b(u_i)=\int_{\mathbb{R}^N}|\nabla u_i|^2-\frac{b_i(Np-\mu)}{2p}\int_{\mathbb{R}^N}\int_{\mathbb{R}^N}\frac{|u_i(x)|^p|u_i(y)|^p}{|x-y|^{\mu}}=0.
\end{align*}
By calculating, we have $\|u_i\|^2=\frac{2(Np-\mu)}{Np-\mu-2}E_{b_i}(s_i^0)+s_i^0$. Hence, we obtain $u_i$ is bounded in $H^1(\mathbb{R}^N)$ for each $i=1,\cdots, k$.

Now, we claim that $\mathcal{S}$ is compact in $(H^1(\mathbb{R}^N))^{ k}$.
If $\mathcal{S}$ contains only finite elements, then $\mathcal{S}$ is compact in $(H^1(\mathbb{R}^N))^{ k}$.
 Otherwise, taking a sequence $\{\mathbf{U_n}\}\in  \mathcal{S}$, then $\{\mathbf{U_n}\}$ is bounded in $(H^1(\mathbb{R}^N))^{ k}$. Up to a subsequence, there exists
 $\mathbf{U}$ such that $\mathbf{U_n}\rightharpoonup\mathbf{U}$ weakly in $(H^1(\mathbb{R}^N))^{ k}$.
 Let $\mathbf{U_n}=(u_{n}^1,\cdots, u_n^{ k})$ and $\mathbf{U}=(u^1,\cdots, u^{ k})$, one has
$u^i_n\rightharpoonup u^i,\ \text{weakly in}\ H^1(\mathbb{R}^N)$ and
  \begin{align*}
 -\Delta u^i+\lambda_0 u^i=b_i\left(\int_{\mathbb{R}^N}\frac{|u^i|^p}{|x-y|^{\mu}}\right)|u^i(x)|^{p-2}u^i,\ \text{ in}\ H^1(\mathbb{R}^N),
 \end{align*}
 for each $i=1,\cdots, k$. Using the Maximum Principle, we obtain that $u^i=0$ or $u^i>0$ for each $i=1,\cdots, k$.
If there exists $i$ such that $u^i=0$, we deduce that $u^i(0)\rightarrow 0$ as $n\rightarrow\infty$.
Then it follows from $\mathcal{P}_{b_i}(u_n^i)=0$ that
\begin{align*}
\int_{\mathbb{R}^N}|\nabla u_n^i|^2=&\frac{b_i(Np-\mu)}{2p}\int_{\mathbb{R}^N}\int_{\mathbb{R}^N}\frac{|u_n^i(x)|^p|u_n^i(y)|^p}{|x-y|^{\mu}}
\leq C\left(\int_{\mathbb{R}^N}|u_n^i|^{2Np/(2N-\mu)}\right)^{(2N-\mu)/N}\\
\leq&C\max\limits_{x\in \mathbb{R}^N}\{|u_n^i|^{2(NP-\mu)/N}\}\left(\int_{\mathbb{R}^N}|u_n^i|^2\right)^{(2N-\mu)/N}\rightarrow 0,
\end{align*}
which is a contradiction. Hence, $u^i>0$.

Next, we prove $v^i_n=u_n^i-u^i\rightarrow0$ in $H^1(\mathbb{R}^N)$ for $i=1,\cdots, k$. By Brezis-Lieb lemma, we have $
s^0_i=\lim\limits_{n\rightarrow\infty}\|v_n^i+u^i\|^2_{L^2(\mathbb{R}^N)}=\|u^i\|^2_{L^2(\mathbb{R}^N)}+\lim\limits_{n\rightarrow
\infty}\|v_n^i\|^2_{L^2(\mathbb{R}^N)}$
and
\begin{align}
 E_{b_i}(s^0_i)=
 \lim\limits_{n\rightarrow+\infty}E_{b_i}(v_n^i)+E_{b_i}(u^i).\label{e15}
 \end{align}
For convince, we define $|u^i|^2_2=s_i$.
If $p\in (\frac{2N-\mu}{N}, \frac{2N-\mu+2}{2})$, it follows from \cite[Lemma 2.5]{Y} that $E_{b_i}(c)<E_{b_i}(c-a)+E_{b_i}(a)$ for any $0<a<c<+\infty$ .
 From (\ref{e15}) we have
 \begin{align*}
 E_{b_i}(s^0_i)=\lim\limits_{n\rightarrow+\infty}E_{b_i}(v_n^i)+E_{b_i}(u^i)\geq  E_{b_i}(s^0_i-s_i)+ E_{b_i}(s_i),
 \end{align*}
which implies $\lim\limits_{n\rightarrow+\infty}E_{b_i}(v_n^i)=0$ and $|u^i|^2_{2}=s^0_i$. Hence we deduce
 $v^i_n\rightarrow0$ in $H^1(\mathbb{R}^N)$ and $u^i\in\overline{S}_{b_i,s_i^0}$.
If $p\in (\frac{2N-\mu+2}{2},\frac{2N-\mu}{N-2})$,  it follows from Lemma \ref{L2-3} $(iv)$ that we have $E_{b_i}(c)<E_{b_i}(a)$ for any $0<a<c$.
Similar to the proof of equality (\ref{e15}), and combining with $\mathcal{P}_{b_i}(u^i)=0$, we have
$0=\lim\limits_{n\rightarrow+\infty}\mathcal{P}_{b_i}(u_n^i)=\mathcal{P}_{b_i}(u^i)+\lim\limits_{n\rightarrow+\infty}\mathcal{P}_{b_i}(v_n^i)=\lim\limits_{n\rightarrow+\infty}
\mathcal{P}_{b_i}(v_n^i)$,
which implies that $\lim\limits_{n\rightarrow+\infty}E_{b_i}(v_n^i)\geq0$.
 Since $\lim\limits_{n\rightarrow+\infty}E_{b_i}(v_n^i)\geq0$ and $u^i\in  S_{s_i,b_i}$, it follows from (\ref{e15}) that
\begin{equation*}
E_{b_i}(s^0_i)=\lim\limits_{n\rightarrow+\infty}E_{b_i}(v_n^i)+E_{b_i}(u^i)\geq E_{b_i}(u^i)\geq E_{b_i}(s_i)\geq E_{b_i}(s^0_i).
\end{equation*}
This indicates that $\lim\limits_{n\rightarrow+\infty}E_{b_i}(v_n^i)=0$ and $|u^i|^2_2=s^0_i$.
 Hence we deduce $v_n\rightarrow 0$ in $H^1(\mathbb{R}^N)$ and $u^i\in\overline{S}_{b_i,s_i^0}$.

From above results, we deduce that $\mathbf{U_n}\rightarrow\mathbf{U}$  in $(H^1(\mathbb{R}^N))^{ k}$ and $\mathbf{U}\in\mathcal{S}$.
\end{proof}

\textbf{Proof of Theorem \ref{T1.3}:} It is a direct results of corollary \ref{L2-5} and Lemma \ref{L2-6}.

\section{The existence of normalized multi-bump solutions}
~~~~In this section, we mainly prove the existence of normalized multi-bump solutions for problem (\ref{e7}).
Define the energy functional $E_{\epsilon}:H^1(\mathbb{R}^N)\rightarrow\mathbb{R}$
is given by
\begin{equation*}
E_{\epsilon}(u):=\frac{1}{2}\int_{\mathbb{R}^N}|\nabla u|^2-\frac{1}{2p}\int_{\mathbb{R}^N}\int_{\mathbb{R}^N}\frac{Q(\epsilon x)Q(\epsilon y)|u(x)|^p|u(y)|^p}{|x-y|^{\mu}}.
\end{equation*}
To find the solutions of problem (\ref{e7}) via a variational method, we need to
look for the critical point of $E_{\epsilon}$ restricted on the constraint
 $ S_1:=\left\{u\in H^1(\mathbb{R}^N)| |u|^2_2=1\right\}$.

\subsection{The concentration-compactness type result}
~~~~We first give some notations.
For the sake of convenience, let
 $b_i=M_i^2$, $i=1,\cdots, k$, where $M_1,\cdots, M_{ k}$ are the positive numbers given in $(Q)$ and
 \begin{align}\label{e16}
 \sigma_0:=\sum_{i=1}^{ k}E_{b_i}(u_i),
  \end{align}
where $(u_1,\cdots,u_{ k})\in\mathcal{S}$. Clearly, $\sigma_0=\sigma_{s^0}$.
For $\tau>0$, let $(\mathcal{A}_i)^{\tau}=\left\{x\in\mathbb{R}^N: dist(x, \mathcal{A}_i)<\tau\right\}$ be a $\tau$-neighbourhood of $\mathcal{A}_i$ for $i=1,\cdots, k$.
Then we choose $\tau>0$ small such that $(\mathcal{A}_i)^{2\tau}\subset\Omega_i$ for $i=1,\cdots, k$.
Fix a cutoff function $\zeta\in C^{\infty}_{0}(\mathbb{R}^N,[0,1])$ such that $0\leq \zeta\leq1$, $\zeta(x)=1$ for $|x|<\tau$ and $\zeta(x)=0$ for $|x|\geq 2\tau$. For convince, we define
$\zeta_{\epsilon}(y)=\zeta(\epsilon y)$, $y\in\mathbb{R}^N$, $(\mathcal{A}_{\epsilon,i})^{\tau}=\frac{1}{\epsilon}(\mathcal{A}_{i})^{\tau}$, $(\mathcal{A}_{\epsilon,i})^{2\tau}=\frac{1}{\epsilon}(\mathcal{A}_{i})^{2\tau}$ and $\Omega_{\epsilon,i}=\frac{1}{\epsilon}\Omega_{i}$ for $i=1,\cdots, k$.
Set
\begin{equation*}
D^{\epsilon}=\left\{\sum_{i=1}^{ k}\zeta_{\epsilon}(x-x_i/{\epsilon})u_i(x-x_i/\epsilon)\bigg| x_i\in(\mathcal{A}_i)^{\tau},u_i\in \overline{S}_{b_i, s_i^0}\right\}
\end{equation*}
and $
D^{\epsilon}_d:=\left\{u\in  S_1\bigg| \inf\limits_{v\in D^{\epsilon}}\|u-v\|\leq d\right\}$
where $0<d<\frac{1}{2}\min\limits_{1\leq i\leq k}\|u_i\|_{L^2(B_1(0))}$,
We will find the solutions to problem (\ref{e7}) in $D^{\epsilon}_d$ for sufficiently small $\epsilon, d>0$.

We denote $
W(u)=\frac{1}{2}|u|^2_2,\ u\in H^1(\mathbb{R}^N)$.
Then we present the following $\epsilon$-dependent concentration-compactness type result, which will give a useful gradient estimate later in Corollary \ref{L3.2}.

\begin{proposition}\label{L3.1}
Suppose that $(Q)$ holds.
 Let $d>0$ and sequences $\{\epsilon_n\}\subset\mathbb{R}^+$, $\{u_n\}\subset D^{\epsilon_n}_{d}$ satisfying
\begin{align}
&\epsilon_n\rightarrow 0^+,\quad \limsup\limits_{n\rightarrow\infty}E_{\epsilon_n}(u_n)\leq \sigma_0, \label{e17} \\
&\|E'_{\epsilon_n}(u_n)+\lambda_nW'(u_n)\|_{H^{-1}}\rightarrow 0.\label{e18}
\end{align}
Then there holds $\lambda_n\rightarrow \lambda_0$, $\lim\limits_{n\rightarrow\infty}E_{\epsilon}(u_n)=\sigma_0$ and
\begin{align}
\|u_n-\sum \zeta_{\epsilon_n}(x-\widetilde{x}_{n}^{i})u_i(x-\widetilde{x}_{n}^{i})\|\rightarrow 0\ \text{and}\ dist(x_{n}^{i}, \mathcal{A}_i)\rightarrow0,\label{e19}
\end{align}
for some $\widetilde{x}_{n}^{i}=\frac{x_{n}^{i}}{\epsilon_n}$ and $u_i\in \overline{S}_{b_i, s_i^0}$, $i=1,\cdots, k$.
Here, $\lambda_0$ and $\sigma_0$ are shown as (\ref{e14}) and (\ref{e16}) respectively.
\end{proposition}

\begin{proof}
For $n$ sufficiently large, by compactness of $\mathcal{S}$, there exist $\mathbf{U}=(u_1,\cdots,u_{ k})\in\mathcal{S}$, $v_n\in H^1(\mathbb{R}^N)$
 and
$x_{n}^{i}\in(\mathcal{A}_i)^{\tau}$ for $i=1,\cdots, k$, passing to a subsequence still denote $\{u_n\}$, such that
\begin{align}
u_n=\sum_{i=1}^{ k}\zeta_{\epsilon_n}(x-\frac{x_{n}^{i}}{\epsilon_n})u_i(x-\frac{x_{n}^{i}}{\epsilon_n})+v_{n},\ \|v_{n}\|\leq d.\label{e20}
\end{align}
For convince, we write $\widetilde{x}_{n}^{i}=\frac{x_{n}^{i}}{\epsilon_n}$.
For any $R\geq1$, one has $
\lim\limits_{n\rightarrow\infty}|u_n(x+\widetilde{x}_{n}^{i})-u_i|_{L^2(B_R(0))}\leq d,$
which implies that for $0<d<\frac{1}{2}\min\limits_{1\leq i\leq k}|u_i|_{L^2(B_1(0))}$, there holds
\begin{align}
0< |u_i|_{L^2(B_R(0))}-d\leq \lim_{n\rightarrow\infty}|u_n(x+\widetilde{x}_{n}^{i})|_{L^2(B_R(0))}\leq d+|u_i|_{L^2(B_R(0))}.\label{e21}
\end{align}
To verify (\ref{e19}), we first denote $u_{n,1}=\sum_{i=1}^{ k} \zeta_{\epsilon_n}(x-\widetilde{x}_{n}^{i})u_n$, $u_{n,2}=u_n-u_{n,1}$ and discuss the properties of $u_{n,1}$ and $u_{n,2}$.

\textbf{Claim 1:} $u_n$ must vanish in $\cup^{ k}_{i=1}B_{2\tau/{\epsilon_n}}(\widetilde{x}_{n}^{i})\setminus B_{\tau/{\epsilon_n}}(\widetilde{x}_{n}^{i})$.

For $R>0$, we define
\begin{align*}
\delta:=\lim\limits_{n\rightarrow\infty}\sup\limits_{y\in \cup^{ k}_{i=1}B_{2\tau/{\epsilon_n}}(\widetilde{x}_{n}^{i})\setminus B_{\tau/{\epsilon_n}}(\widetilde{x}_{n}^{i})}
\int_{B_R(y)}|u_n|^2dx.
\end{align*}
We claim that $\delta=0$. Otherwise, for $\epsilon_n>0$ small, there exist $y_{n,i}\in B_{2\tau/{\epsilon_n}}(\widetilde{x}_{n}^{i})\setminus B_{\tau/{\epsilon_n}}(\widetilde{x}_{n}^{i})$ for some $i$, such that
\begin{align*}
\int_{B_R(y_{n,i})}|u_n|^2>C_0>0.
\end{align*}
Since $y_{n,i}\in B_{2\tau/{\epsilon_n}}(\widetilde{x}_{n}^{i})\setminus B_{\tau/{\epsilon_n}}(\widetilde{x}_{n}^{i})$, it holds $
\|\zeta_{\epsilon_n}(\cdot-\widetilde{x}_{n}^{i})u_i(\cdot-\widetilde{x}_{n}^{i})\|_{B_R(y_{n,i})}=o_n(1)$.
It follows from (\ref{e20}) that $
\|u_n-\zeta_{\epsilon_n}(\cdot-\widetilde{x}_{n}^{i})u_i(\cdot-\widetilde{x}_{n}^{i})\|_{B_R(y_{n,i})}\leq 2d.$
Thus, for sufficiently small $d$, we obtain a contradiction. Similarly, according to Lion's lemma (See \cite{Lions}, Lemma 1.1), we can derive that
\begin{align}
\delta:=\lim\limits_{n\rightarrow\infty}
\int_{ \cup^{ k}_{i=1}B_{2\tau/{\epsilon_n}}(\widetilde{x}_{n}^{i})\setminus B_{\tau/{\epsilon_n}}(\widetilde{x}_{n}^{i})}|u_n|^qdx=0,\ \text{for}\ 2<q<2^{\ast}.
\label{e22}
\end{align}

\textbf{Claim 2:} $E_{\epsilon_n}(u_n)\geq  E_{\epsilon_n}(u_{n,1})+ E_{\epsilon_n}(u_{n,2})+o_n(1)$.

It follows from (\ref{e21}) and (\ref{e22}) that
\begin{align*}
 \int_{\mathbb{R}^N}\int_{\mathbb{R}^N}\frac{Q(\epsilon_n x)|u_n(x)|^pQ(\epsilon_n y)|u_n(y)|^p}{|x-y|^{\mu}}
 =&\int_{\mathbb{R}^N}\int_{\mathbb{R}^N}\frac{Q(\epsilon_n x)|u_{n,1}(x)|^pQ(\epsilon_n y)|u_{n,1}(y)|^p}{|x-y|^{\mu}}\\&+
 \int_{\mathbb{R}^N}\int_{\mathbb{R}^N}\frac{Q(\epsilon_n x)|u_{n,2}(x)|^pQ(\epsilon_n y)|u_{n,2}(y)|^p}{|x-y|^{\mu}}+o_n(1).
 \end{align*}
Then we have
 \begin{align*}
 E_{\epsilon_n}(u_n)=&E_{\epsilon_n}(u_{n,1})+E_{\epsilon_n}(u_{n,2})+\int_{\mathbb{R}^N}\sum_{i=1}^{ k}\zeta_{\epsilon_n}(x-\widetilde{x}_{n}^{i})
 \left(1-\sum_{i=1}^{ k}\zeta_{\epsilon_n}(x-\widetilde{x}_{n}^{i})\right)
 |\nabla u_n|^2+o_n(1)\\
 \geq&E_{\epsilon_n}(u_{n,1})+ E_{\epsilon_n}(u_{n,2})+o_n(1).
 \end{align*}

 \textbf{Claim 3:} $E_{\epsilon_n}(u_{n,2})\geq o_n(1)$.

From (\ref{e20}), we conclude that $\|u_{n,2}\|\leq 2d$ for $n$ sufficiently large.
A direct calculation gives that
 \begin{align*}
E_{\epsilon_n}(u_{n,2})= & \frac{1}{2}\int_{\mathbb{R}^N}|\nabla u_{n,2}|^2-\frac{1}{2p}\int_{\mathbb{R}^N}\int_{\mathbb{R}^N}\frac{|Q(\epsilon x)u_{n,2}|^p|Q(\epsilon y)u_{n,2}|^p}{|x-y|^{\mu}}\\ \geq &\frac{1}{2}\int_{\mathbb{R}^N}|\nabla u_{n,2}|^2-
C\left(\int_{\mathbb{R}^N}|u_{n,2}|^{2Np/(2N-\mu)}\right)^{(2N-\mu)/N}\\
\geq& \frac{1}{2}\int_{\mathbb{R}^N}|\nabla u_{n,2}|^2-C_{\beta}\left(\int_{\mathbb{R}^N}|\nabla u_{n,2}|^2\right)^p-
\beta\left(\int_{\mathbb{R}^N}|u_{n,2}|^2\right)^p.
\end{align*}
 Taking $ d$ and $\beta$ small enough, we can deduce that $E_{\epsilon_n}(u_{n,2})\geq o_n(1)$.

\textbf{Claim 4:} $\|u_{n,1}-\sum\limits_{i=1}^{ k}\zeta_{\epsilon_n}(\cdot-\widetilde{x}_{n}^{i})u_i(\cdot-\widetilde{x}_{n}^{i})\|=o_n(1)$.

By (\ref{e17}), we show that $
|\lambda_n|=\left|\lambda_n\int_{\mathbb{R}^N}|u_n|^2\right|\leq \|u_n\|^2+C\|u_n\|^p\leq C$,
which implies that $\lambda_n$ is bounded.
Recall that $u_n$ does not vanish near $\widetilde{x}_{n}^{i}$, then we define $
\phi_{n}^{i}=\zeta_{\epsilon_n}(x)u_n(x+\widetilde{x}_{n}^{i})$.
Then up to a subsequence, $\epsilon_n\widetilde{x}_{n}^{i}\rightarrow x_i\in (\mathcal{A}_i)^{\tau}$
and $\lambda_n\rightarrow\lambda$ in $\mathbb{R}^N$,
$\phi_{n}^{i}\rightharpoonup \phi_i$ weakly in $H^1(\mathbb{R}^N)$.
Moreover, $\phi_i\neq 0$ satisfying
\begin{align*}
-\Delta \phi_i+\lambda \phi_i=\left(Q(x_i)\right)^2\int_{\mathbb{R}^N}\frac{|\phi_i(y)|^p}{|x-y|^{\mu}}dy|\phi_i(x)|^{p-2}\phi_i(x),\ x\in\mathbb{R}^N.
\end{align*}
By the Poho\v{z}aev identity, we deduce $\lambda>0$.

Now, we prove $\phi_{n}^{i}\rightarrow \phi_i$ strongly in $H^1(\mathbb{R}^N)$.  First of all, we proof that $\phi_{n}^{i}\rightarrow \phi_i$ strongly in $L^q$ for any $q\in(2, 2^{\ast})$. Otherwise,
there exist $R>0$ and a sequence $\{x_n\}\subset\mathbb{R}^N$ with $x_n\in B(\widetilde{x}_{n}^{i}, 2\tau/{\epsilon_n})$ such that
$|x_n-\widetilde{x}_{n}^{i}|\rightarrow \infty$ and
\begin{align*}
\varliminf\limits_{n\rightarrow\infty}\int_{B_R(x_n)}|\zeta_{\epsilon_n}(x-\widetilde{x}_{n}^{i})u_n(x)|^2>0.
\end{align*}
Let $\epsilon_nx_n\rightarrow x_0\in B_{2\tau}(x_i)$ as $n\rightarrow\infty$ and
define $\widetilde{\phi}_{n}^{i}=\zeta_{\epsilon_n}(x-\widetilde{x}_{n}^{i}+x_n)u_n(x+x_n)$, then up to a subsequence, $
\widetilde{\phi}_{n}^{i}\rightharpoonup \widetilde{\phi}_i\neq 0$
weakly in $H^1(\mathbb{R}^N)$ and satisfying
\begin{align*}
-\Delta \widetilde{\phi}_i+\lambda \widetilde{\phi}_i=
\left(Q(x_0)\right)^2\int_{\mathbb{R}^N}\frac{|\widetilde{\phi}_i(y)|^p}{|x-y|^{\mu}}dy|\widetilde{\phi}_i(x)|^{p-2}\widetilde{\phi}_i(x),\ x\in\mathbb{R}^N
\end{align*}
Similar to the argument of Claim $1$,  we get a contradiction. So, $\phi_{n}^{i}\rightarrow \phi_i$ strongly in $L^q(\mathbb{R}^N)$ for any $q\in(2,2^{\ast})$. Then we have
\begin{align*}
&\lim\limits_{n\rightarrow \infty}\int_{\mathbb{R}^N}\int_{\mathbb{R}^N}\frac{|Q(\epsilon_n x)\phi_{n}^{i}(x-\widetilde{x}_{n}^{i})|^p|Q(\epsilon_n y)\phi_{n}^{i}(y-\widetilde{x}_{n}^{i})|^p}{|x-y|^{\mu}}\\=&
\lim\limits_{n\rightarrow \infty}\int_{\mathbb{R}^N}\int_{\mathbb{R}^N}\frac{|Q(\epsilon_n (x+\widetilde{x}_{n}^{i}))
\phi_{n}^{i}(x)|^p|Q(\epsilon_n (y+\widetilde{x}_{n}^{i}))\phi_{n}^{i}(y)|^p}{|x-y|^{\mu}}\\=&\left(Q(x_i)\right)^2
\int_{\mathbb{R}^N}\int_{\mathbb{R}^N}\frac{|\phi_{i}(x)|^p|\phi_{i}(y)|^p}{|x-y|^{\mu}}
\end{align*}
and $
\sum_{i=1}^{ k}|\phi_i|_2^2=\lim\limits_{n\rightarrow\infty}\sum_{i=1}^{ k}|\zeta_{\epsilon_n}(x-\widetilde{x}_{n}^{i})u_n(x)|^2_2=1.$
It follows from the weak convergence of $\phi_{n}^{i}$ to $\phi_i$ in $H^1(\mathbb{R}^N)$ that
\begin{align*}
E_{\epsilon_n}(u_{n,1})= & \sum_{i=1}^{ k}\frac{1}{2}\int_{\mathbb{R}^N}|\nabla \phi_{n}^{i}(x-\widetilde{x}_{n}^{i})|^2-\frac{1}{2p}\int_{\mathbb{R}^N}\int_{\mathbb{R}^N}\frac{|Q(\epsilon_n x)
\phi_{n}^{i}(x-\widetilde{x}_{n}^{i})|^p|Q(\epsilon_n y)\phi_{n}^{i}(y-\widetilde{x}_{n}^{i})|^p}{|x-y|^{\mu}}\\ \geq &\sum_{i=1}^{ k}\frac{1}{2}\int_{\mathbb{R}^N}|\nabla \phi_i|^2
-\frac{1}{2p}Q^2(x_i)\int_{\mathbb{R}^N}\int_{\mathbb{R}^N}\frac{|\phi_{i}(x)|^p|\phi_{i}(y)|^p}{|x-y|^{\mu}}+o_n(1)
\\ =& \sum_{i=1}^{ k} E_{Q^2(x_i)}(\phi_i)+o_n(1).
\end{align*}
From Lemmas \ref{L2-4} $(i)$ and \ref{L2-6}, we deduce that
\begin{align*}
\sum\limits_{i=1}^{ k}E_{Q^2(x_i)}(\phi_i)\geq E_1(1)\left[\sum\limits_{i=1}^{ k}(Q(x_i))^{-2/(p-1)}\right]^{\frac{-2p+2}{-Np+2N-\mu+2}}\geq \sigma_0.
\end{align*}
Then combining with  (\ref{e17}) and the conclusion of Claim 2, we obtain that
\begin{equation*}
\sigma_0\leq\lim\limits_{n\rightarrow\infty}
E_{\epsilon_n}(u_{n,1})\leq \limsup\limits_{n\rightarrow\infty}
E_{\epsilon_n}(u_{n})\leq \sigma_0.
\end{equation*}
Hence, we deduce that $\phi_{n}^{i}\rightarrow \phi_i$ strongly in $H^1(\mathbb{R}^N)$ and $Q(x_i)=M_i$, $\lambda=\lambda_0$, $\phi_i=u_i$ for some $u_i\in\overline{S}_{b_i, s_i^0}$, where $i=1,\cdots, k$.  By calculating, $
\|u_{n,1}-\sum\limits_{i=1}^{ k}\zeta_{\epsilon_n}(\cdot-\widetilde{x}_{n}^{i})u_i(\cdot-\widetilde{x}_{n}^{i})\|=o_n(1)$. From the above results, we complete the proof.
\end{proof}

\begin{corollary}\label{L3.2}
Suppose that $(Q)$ holds.
Let $d>0$ small enough. Then constants $c,\epsilon_0>0$ exists such that $\|E'_{\epsilon}|_{ S_1}(u)\|_{T^*_u S_1}\geq c$
for $u\in E_{\epsilon}^{\sigma_{\epsilon}}\cap (D^{\epsilon}_d\setminus D^{\epsilon}_{d/2})$ and $\epsilon\in (0,\epsilon_0)$,
where $E_{\epsilon}^{\sigma_{\epsilon}}=\{u\in H^1(\mathbb{R}^N):\ E_{\epsilon}\leq \sigma_{\epsilon}$\}.
\end{corollary}

Set the tangent space to $ S_1$ at $u\in  S_1$ by $
T_u S_1=\{v\in H^1(\mathbb{R}^N)|\ \int_{\mathbb{R}^N} uv=0\}$
and $T^*_u S_1$ the dual space of $T_{u} S_1$ with the dual norm $
\|f\|_{T^*_u S_1}=\sup\limits_{v\in T_u S_1,\|v\|=1}f(v)=\min\limits_{\lambda\in\mathbb{R}}\|f-\lambda W'(u)\|_{H^{-1}}$.
Next we show that $E_{\epsilon}$ satisfies the Palais-Smale condition in $D^{\epsilon}_d$ for $\epsilon>0$ fixed.

\begin{proposition}\label{L3.3}
Suppose that $(Q)$ holds.
For each $\sigma\in\mathbb{R}^+$,  there exists $\epsilon_{\sigma}>0$ such that for any $\epsilon\in (0,\epsilon_{\sigma})$, if a sequence $\{v_{n}\}\subset D^{\epsilon}_d$ satisfies
\begin{align}
E_{\epsilon}(v_{n})\leq\sigma,\ \lim\limits_{n\rightarrow\infty}\|E'_{\epsilon}|_{ S_1}(v_{n})\|_{T^*_{v_{n}} S_1}=0,\label{e23}
\end{align}
then up to a subsequence, there exists $v_0\in H^1(\mathbb{R}^N)$ such that $v_{n}\rightarrow v_0$ in $H^1(\mathbb{R}^N)$.
\end{proposition}

\begin{proof}
It follows from the definition of $D^{\epsilon}_d$ and the compactness of $\mathcal{S}$ that we may assume that up to a subsequence,
$v_{n}=\sum_{i=1}^{ k}\zeta_{\epsilon}(x-x_{n,\epsilon}^i)u_i(x-x_{n,\epsilon}^i)+v_{n,\epsilon}$ with $\epsilon x_{n,\epsilon}^i\in(\mathcal{A}_i)^{\tau}$
, $(u_1,\cdots,u_{ k})\in \mathcal{S}$ and $\|v_{n,\epsilon}\|\leq d$. Clearly, we can easily deduce that $v_{n}$ is bounded.
From the assumption (\ref{e23}), we deduce that for some $\lambda_{n, \epsilon}\in\mathbb{R}$, there holds $
E'_{\epsilon}(v_{n})+\lambda_{n,\epsilon}W'(v_{n})=o_n(1)$ in $H^{-1}$.
That is,
for any $\eta\in H^1(\mathbb{R}^N)$, it holds
\begin{align}
\int_{\mathbb{R}^N}\left(\nabla v_{n}\nabla \eta+\lambda_{n,\epsilon} v_{n}\eta\right)-\int_{\mathbb{R}^N}\int_{\mathbb{R}^N}\frac
{Q(\epsilon x)Q(\epsilon y)|v_{n}(x)|^p|v_{n}(y)|^{p-2}\eta v_{n}}{|x-y|^{\mu}}=o_n(1).\label{e24}
\end{align}
It follows from (\ref{e24}) that
\begin{align*}
|\lambda_{n,\epsilon}|=&|\lambda_{n,\epsilon}W'(v_{n})v_{n}|=|E'_{\epsilon}(v_{n})+o_n(1)|\\
\leq& C\left(\|v_{n}\|^2+
\left(\int_{\mathbb{R}^N}|v_{n}|^p)^{2N/(2N-\mu)}\right)^{(2N-\mu)/N}\right)\\
\leq&C\left(\|v_{n}\|^2+\|v_{n}\|^{2p}\right)
\end{align*}
Note that for $d<\frac{1}{2}\min_{1\leq i\leq k}\|u_i\|_{H^1(B_1(0))}$, it holds
\begin{equation*}
\liminf_{n\rightarrow\infty}\|v_{n}(\cdot+x_{n,\epsilon}^{i})\|_{H^1(B_1(0))}\geq \|u_i\|_{H^1(B_1(0))}-d>0.
\end{equation*}
Thus we assume that up to a subsequence, $\lambda_{n,\epsilon}\rightarrow\lambda_{\epsilon}$ in $\mathbb{R}$ and $v_{n}\rightharpoonup v_0=
\sum_{i=1}^{ k}
\zeta(\epsilon x-x_{\epsilon}^i)u_{i}(x-x_{\epsilon}^i)+v_{\epsilon}$ in $H^1(\mathbb{R})\setminus\{0\}$ with $x_{n,\epsilon}^i\rightarrow x_{\epsilon}^i\in (\mathcal{A}_{\epsilon,i})^{\tau}$ and $v_{n,\epsilon}\rightharpoonup v_{\epsilon}$ in $H^1(\mathbb{R}^N)$. Moreover, $(\lambda_{\epsilon}, v_0)$ satisfies that for any $\eta\in H^1(\mathbb{R}^N)$,
\begin{align}
\int_{\mathbb{R}^N}\left(\nabla v_0\nabla \eta+\lambda_{\epsilon} v_0\eta\right)-\int_{\mathbb{R}^N}\int_{\mathbb{R}^N}\frac
{Q(\epsilon x)Q(\epsilon y)|v_0(x)|^p|v_0(y)|^{p-2}\eta v_0}{|x-y|^{\mu}}=0.\label{e25}
\end{align}

Now we prove $\lambda_{\epsilon}>0$ for $\epsilon$ sufficiently small. Suppose to the contrary, we assume that  up to a subsequence, $\lambda_{\epsilon}\rightarrow\lambda\leq 0$ as $\epsilon\rightarrow0$. Since $v_0$ is
bounded in $H^1(\mathbb{R}^N)$, we may assume that $v_0(\cdot+x_{\epsilon}^i)\rightharpoonup v$.
According to the following inequality
\begin{align*}
\liminf\limits_{\epsilon\rightarrow0}\|v_0(\cdot+x_{\epsilon}^i)\|_{H^1(B_1(0))}\geq \|u_i\|_{H^1(B_1(0))}-d>0.
\end{align*}
We have $v\neq0$. Denote $\psi=\eta(\cdot-x_{\epsilon}^1)$ in (\ref{e25}) for each $\eta\in H^1(\mathbb{R}^N)$ and take limits as $\epsilon\rightarrow0$.
We conclude that $v$ is a nontrivial solution to $-\Delta u+\lambda u=(M_0)^2\left(\int_{\mathbb{R}^N}\frac{|u|^p}{|x-y|^{\mu}}\right)|u|^{p-2}u$ for some $M_0>0$.
By the Poho\v{z}aev  identity, we deduce $\lambda>0$ which is a contradiction. Thus, we conclude that there exist $\epsilon_{\sigma}$ and $\lambda_{\sigma}$
 such that $\lambda_{\epsilon}>\lambda_{\sigma}$ for each $(0,\epsilon_{\sigma})$.

Combining (\ref{e24})-(\ref{e25}) and replacing $\eta$ by
$v_n-v_0$, one has
\begin{align*}
&\int_{\mathbb{R}^N}\left(|\nabla (v_n-v_0)|^2+\
\lambda_{\epsilon}(v_n-v_0)^2+(\lambda_{n,\epsilon}-\lambda_{\epsilon})v_{n}(v_n-v_0)\right)
\\-&\int_{\mathbb{R}^N}\int_{\mathbb{R}^N}\frac{Q(\epsilon y)Q(\epsilon x)(
|v_{n}(x)|^p|v_{n}(y)|^{p-2} v_{n}-
|v_0(x)|^p|v_0(y)|^{p-2} v_0)(v_n-v_0)}{|x-y|^{\mu}}
=o_n(1).
\end{align*}
Since $\lambda_{n,\epsilon}\rightarrow\lambda_{\epsilon}>0$ in $\mathbb{R}$ and
$v_{n}\rightharpoonup v_0$ as $n\rightarrow\infty$, we deduce that
\begin{align*}
&\min\{1,\lambda_{\sigma}/2\}\|v_n-v_0\|^2
\\ \leq& \int_{\mathbb{R}^N}\int_{\mathbb{R}^N}\frac{Q(\epsilon y)Q(\epsilon x)\left(
|v_{n}(x)|^p|v_{n}(y)|^{p-2} v_{n}-
|v_0(x)|^p|v_0(y)|^{p-2} v_0\right)(v_n-v_0)}{|x-y|^{\mu}}
+o_n(1)\\
=&\int_{\mathbb{R}^N}\int_{\mathbb{R}^N}\frac{
Q(\epsilon y)Q(\epsilon x)(|v_{n}(x)|^p|v_{n}(y)|^{p}-|v_0(x)|^p|v_0(y)|^{p})}{|x-y|^{\mu}}+o_n(1)\\
=&\int_{\mathbb{R}^N}\int_{\mathbb{R}^N}\frac{
Q(\epsilon y)Q(\epsilon x)|v_{n}(x)-v_0(x)|^p|v_{n}(y)-v_0(y)|^{p}}{|x-y|^{\mu}}+o_n(1).
\end{align*}
For $R>0$ sufficiently large such that $\cup^{ k}_{i=1}\Omega_{\epsilon,i}\subset B_R(0)$, the Classical Rellich-Kondrachov embedding theorem and Sobolev inequality imply that
\begin{align*}
&\int_{\mathbb{R}^N}\int_{\mathbb{R}^N}\frac{
Q(\epsilon y)Q(\epsilon x)|v_n(x)-v_0(x)|^p|v_{n}(y)-v_0(y)|^{p}}{|x-y|^{\mu}}\\
\leq& C\left(\int_{\mathbb{R}^N}|v_n(x)-v_0(x)|^{2Np/(2N-\mu)}\right)^{(2N-\mu)/N}\\
\leq &C\left(\int_{B_R(0)}|v_n(x)-v_0(x)|^{2Np/(2N-\mu)}+\int_{\mathbb{R}^N\setminus B_R(0)}|v_{n}(x)
-v_0(x)|^{2Np/(2N-\mu)}\right)^{\frac{(p-1)(2N-\mu)}{Np}}\|v_n-v_0\|^2\\
\leq& C\left(\int_{\mathbb{R}^N\setminus B_R(0)}|v_n(x)-v_0(x)|^{2Np/(2N-\mu)}\right)^{\frac{(p-1)(2N-\mu)}{Np}}\|v_n-v_0\|^2+o_n(1)\\
\leq &C\|v_n-v_0\|^{2(p-1)}_{H^1(\mathbb{R}^N\setminus B_R(0))}\|v_n-v_0\|^2+o_n(1).
\end{align*}
Controlling $\epsilon_{\sigma}$ sufficiently small such that
$C\|v_n-v_0\|^{2(p-1)}_{H^1(\mathbb{R}^N\setminus B_R(0))}<\frac{1}{2}\min\{1,\lambda_{\sigma}/2\}$,
 there holds
 $\lim\limits_{n\rightarrow\infty}\|v_n-v_0\|=0$. Thus we complete the proof.
\end{proof}

\subsection{The proof of Theorem \ref{T1.5}}
In this section, we complete the proof of Theorem \ref{T1.5}.
Fix $d>0$ such that Proposition \ref{L3.1} and Corollary \ref{L3.2} hold. Then we have the following results.
\begin{proposition} \label{L3.4}
Suppose that $(Q)$ holds. Then
there exists a $\overline{\epsilon}_{\sigma}>0$ such that for any $\epsilon\in(0, \overline{\epsilon}_{\sigma})$, $E_{\epsilon}|_{S_1}$ has a critical point $v_{\epsilon}$ in $D^{\epsilon}_d$. Furthermore, it holds $\lim\limits_{\epsilon\rightarrow0}E_{\epsilon}(v_{\epsilon})=\sigma_0$ and the corresponding Lagrange multiplier $\lambda_{\epsilon}$ satisfis
\begin{align}
\ \lim\limits_{\epsilon\rightarrow0}\lambda_{\epsilon}=\lambda_0,\ \text{and}\ E'_{\epsilon}(v_{\epsilon})+\lambda_{\epsilon}W'(v_{\epsilon})=0,\label{e26}
\end{align}
where $\lambda_0$ and $\sigma_0$ are shown as (\ref{e14}) and (\ref{e16}) respectively..
\end{proposition}

\textbf{Proof of Proposition \ref{L3.4}:}
Suppose to the contrary, we assume that there exists
$\epsilon_n\rightarrow 0$ such that for any sequence $\{\sigma_{\epsilon_n}\}$, there holds
 $\sigma_{\epsilon_n}\rightarrow \sigma_0$ and $E_{\epsilon_n}$ admits no critical points in
$\{u\in D^{\epsilon_n}_{d}| E_{\epsilon_n}\leq \sigma_{\epsilon_n}\}$. For convince, let $\epsilon=\epsilon_n$. It follows from Proposition \ref{L3.3} and Corollary
\ref{L3.2} that there exist $\sigma_1>0, \omega>0$ independent of $\epsilon$ and $\omega_{\epsilon}>0$ such that
\begin{align}
&\|E'_{\epsilon}|_{ S_1}(u)\|_*\geq \omega_{\epsilon},\ \text{for}\ u\in [\sigma_0-2\sigma_1\leq E_{\epsilon}\leq \sigma_{\epsilon}]\cap D^{\epsilon}_d,\notag\\
&\|E'_{\epsilon}|_{ S_1}(u)\|_*\geq \omega,\ \text{for}\ u\in [\sigma_0-2\sigma_1\leq E_{\epsilon}\leq \sigma_{\epsilon}]\cap (D^{\epsilon}_d\setminus D^{\epsilon}_{d/2}),\label{e27}
\end{align}
where
we define $[\sigma_0-2\sigma_1\leq E_{\epsilon}\leq \sigma_{\epsilon}]=\{u\in H^1(\mathbb{R}^N)|\sigma_0-2\sigma_1\leq E_{\epsilon}(u)\leq \sigma_{\epsilon}\}$. Similarly, we also denote  $[\sigma_0-2\sigma_0\leq E_{\epsilon}]$, $[E_{\epsilon}\leq \sigma_{\epsilon}]$ etc. From the conditions (\ref{e27}), we  can construct two gradient flows as follows:
\begin{lemma} \label{L3.5}
Assume that the conditions (\ref{e27}) hold. Then for any positive constant $\sigma<\min\left\{\sigma_1,\frac{d\omega}{16}\right\}$,  an $\epsilon_{\sigma}>0$ exists such that
 for any $\epsilon\in(0,\epsilon_{\sigma})$, there exists a deformation $\theta: S_{1}\rightarrow  S_{1}$ and satisfies the following properties:

$(i)$ $\theta(u)=u$, if $u\in S_{1}\setminus\left(D^{\epsilon}_d\cap\left[\sigma_0-2\sigma\leq E_{\epsilon}\right]\right)$.

$(ii)$ $E_{\epsilon}(\theta(u))\leq E_{\epsilon}(u)$, if $u\in S_1$.

$(iii)$ $\theta(u)\in D^{\epsilon}_d\cap\left[ E_{\epsilon}\leq  \sigma_0-\sigma\right]$, if $u\in D^{\epsilon}_{d/4}\cap
\left[E_{\epsilon}\leq \sigma_{\epsilon}\right]$.
\end{lemma}

According to the compactness of $\mathcal{S}$, for each $u\in D^{\epsilon}_d$,  there exist $(u_1,\cdots,u_{ k})\in\mathcal{S}$ and $x_i\in(\mathcal{A}_i)^{\tau}$ such that $u=\zeta_{\epsilon}(x-x_i/\epsilon)u_i(x-x_i/\epsilon)+v$ where $\|v\|\leq d$, $v\in H^1(\mathbb{R}^N)$. This implies $\|u\|_{H^1(\cup_{i=1}^{ k}(\mathcal{A}_{\epsilon,i})^{\tau})^c}\leq 2d$ for $\epsilon$ small.
Define $
F_{\epsilon}^d=\left\{u\in H^1(\mathbb{R}^N):\ \inf\limits_{v\in D^{\epsilon}}\|u-v\|\leq d\right\}$.
Then we construct another gradient flow.
\begin{proposition}\label{L3.6}
Let $\epsilon>0$ be sufficiently small and fixed.
 Then there exists a map $J_{\epsilon}: D^{\epsilon}_d\rightarrow F^{\epsilon}_d$ such that

$(i)$ $J_{\epsilon}$ is a continuous function.

$(ii)$ $J_{\epsilon}(u)=u$ if $u(x)=0$ for all $x\in\left(\cup_{i=1}^{ k}(\mathcal{A}_{\epsilon,i})^{\tau}\right)^c$.

$(iii)$ For all $u\in D^{\epsilon}_d$, there holds
\begin{align}
&J_{\epsilon}(u)(x)=u(x)\ \text{for}\ x\in\cup_{i=1}^{ k}(\mathcal{A}_{\epsilon,i})^{\tau},\label{e28}\\
&|\nabla J_{\epsilon}(u)|^2_{L^2(\cup_{i=1}^{ k}\Omega_{\epsilon,i})^c}\leq c_{\epsilon},\label{e29}\\
&E_{\epsilon}(J_{\epsilon}(u))\leq E_{\epsilon}(u),\label{e30}
\end{align}
where $ c_{\epsilon}>0$ is independent of $u$ and satisfies $\lim\limits_{\epsilon\rightarrow0}c_{\epsilon}\rightarrow0$.
\end{proposition}

The proof of Lemma \ref{L3.5} and proposition \ref{L3.6}  are given in the appendix.

Next, we verify the existence of critical points based on the two gradient flows established in Lemma \ref{L3.5} and proposition \ref{L3.6}, combined with degree theory. Through the energy geometry of the limit system (\ref{e8}), we can discuss the mass subcritical and mass supercritical cases respectively. That is, the case of  $p\in \left(\frac{2N-\mu}{N},\frac{2N-\mu+2}{N}\right)$ or $p\in \left(\frac{2N-\mu+2}{N},\frac{2N-\mu}{N-2}\right)$.

\textbf{Case 1: $L^2$-subcritical case $p\in \left(\frac{2N-\mu}{N},\frac{2N-\mu+2}{N}\right)$.}

Fix $x_i\in \mathcal{A}_{i}$ and set
$x_{\epsilon}^i=\frac{1}{\epsilon} x_i$ for $i=1,2,\cdots, k$.
Then we define the $( k-1)$-dimensional initial path by $
L_{\epsilon,s}(x)=T_{\epsilon}\sum_{i=1}^{ k}(\zeta_{\epsilon} u_{s_i})(x-x_{\epsilon}^i)\quad \text{for}\ s=(s_1, \cdots,s_{ k})\in \mathcal{S}_{k}$,
where $u_{s_i}\in \overline{S}_{b_i,s_i}$ and $
T_{\epsilon}:=\left|\sum^{ k}_{i=1}(\zeta_{\epsilon}u_{s_i})(\cdot-x_{\epsilon}^i)\right|^{-1}_{2}$.
Define $
K_{\nu}:=\{s\in S^+_{ k-1}|\ |s-s^0|\leq\nu\}$
and fix $\nu>0$ small enough (independent of $\epsilon$ small) such that $
L_{\epsilon,s}\in D^{\epsilon}_{d/4}\quad \text{for}\ s\in K_{\nu} $
and $T_{\epsilon}\rightarrow 1$ as $\epsilon\rightarrow 0$ uniformly in $K_{\nu}$.
Here, $s^0$ is given as (\ref{e13}).
We denote
\begin{align*}
\sigma_{\epsilon}=\sup\limits_{s\in K_{\nu}} E_{\epsilon}(L_{\epsilon,s}).
\end{align*}
Clearly, it holds that $\lim\limits_{\epsilon\rightarrow0}\sigma_{\epsilon}=\sigma_0$.
Combining with Lemma \ref{L2-4} $(ii)$, we easily obtain the following conclusions:
there is $\sigma\in \left(0,\min\{\sigma_1,\frac{d\omega}{16}\}\right)$ such that for any $\epsilon\in (0,\epsilon_{\sigma})$,
\begin{align}\label{e31}
\sup\limits_{s\in \partial K_{\nu}}E_{\epsilon}(L_{\epsilon,s})<\sigma_0-2\sigma.
\end{align}
It follows from lemma \ref{L3.5} and the inequality (\ref{e31}) that
\begin{align*}
&\theta\circ L_{\epsilon,s}=L_{\epsilon,s},\quad \text{if}\ s\in\partial K_{\nu},\\
&E_{\epsilon}\circ\theta\circ L_{\epsilon,s}\leq \sigma_0-\sigma\ \text{and}\ \theta\circ L_{\epsilon,s}\in D^{\epsilon}_d,\ \text{if}\ s\in K_{\nu}.
\end{align*}
Then from Proposition \ref{L3.6}, we obtain that
\begin{align}\label{e32}
\begin{cases}
E_{\epsilon}\circ J_{\epsilon}\circ \theta\circ L_{\epsilon,s}\leq \sigma_0-\sigma,\
\text{if}\ s\in K_{\nu},\\
\|\nabla J_{\epsilon}(\theta(L_{\epsilon,s}))\|_{L^2(\cup_{i=1}^{ k}\Omega_{\epsilon,i})^c}= o_{\epsilon}(1),\ \\
J_{\epsilon}\circ\theta\circ L_{\epsilon,s}(x)=\theta\circ L_{\epsilon,s}(x),\quad \text{if}\
x\in \cup_{i=1}^{ k}(\mathcal{A}_{\epsilon,i})^{\tau},\\
J_{\epsilon}\circ\theta\circ L_{\epsilon,s}(x)=L_{\epsilon,s}(x),\quad \text{if}\ s\in\partial K_{\nu}\ \text{and}\
x\in \cup_{i=1}^{ k}(\mathcal{A}_{\epsilon,i})^{\tau}.
\end{cases}
\end{align}
For each $i=1,\cdots, k$, we define
\begin{equation*}
H_{\epsilon,i}(v)=\left(\int_{\Omega_{\epsilon,i}}|v|^2\right)\left(\sum\limits_{i=1}^{ k}\int_{\Omega_{\epsilon,i}}|v|^2\right)^{-1}
\quad \text{for}\ v\in H^1(\mathbb{R}^N).
\end{equation*}
Then we prove that there exists $s^1\in K_{\nu}$ such that $H_{\epsilon,i}(J_{\epsilon}(\theta(L_{\epsilon,s^1})))=s_i^0, (i=1,\cdots,k)$.
For each $s\in K_{\nu}$, we define $
\rho_{\epsilon}(s):=(H_{\epsilon,1}\circ J_{\epsilon}\circ\theta\circ L_{\epsilon,s},\cdots,H_{\epsilon, k}\circ J_{\epsilon}\circ\theta\circ L_{\epsilon,s})$.
Note that
\begin{equation*}
H_{\epsilon,i}\circ L_{\epsilon,s}\rightarrow |u_{s_i}|^2_2=s_i\quad \text{as}\ \epsilon\rightarrow 0\ \text{uniformly for every}\ s\in  K_{\nu}.
\end{equation*}
Combining with (\ref{e32}),
we have $s^1\not\in \rho_{\epsilon}(\partial K_{\nu})$ and the Brower degree $
\deg(\rho_{\epsilon},K_{\nu},s^0)=\deg(id, K_{\nu}, s^0)=1$ for $\epsilon$ small.
This implies there exists $s^1\in K_{\nu}$ such that $H_{\epsilon,i}(J_{\epsilon}(\theta(L_{\epsilon,s^1})))=s^0_i=|u_i|_2$.
Now denote
\begin{align*}
\phi_{\epsilon}:=
\begin{cases}
J_{\epsilon}(\theta(L_{\epsilon,1}))\quad &\text{if}\  k=1,\\
J_{\epsilon}(\theta(L_{\epsilon,s^1}))\quad &\text{if}\  k\geq 2,
\end{cases}\quad
\phi_{i,\epsilon}:=\zeta_{\epsilon,i}\phi_{\epsilon},
\end{align*}
where $\zeta_{\epsilon,i}\in C^{\infty}_0(\Omega'_{\epsilon,i},[0,1])$ is a cut-off function such that $\zeta_{\epsilon,i}=1$ on
$\Omega_{\epsilon,i}$
and $|\nabla \zeta_{\epsilon,i}|\leq C\epsilon$ for each $i=1,2,\cdots, k$ and some $C>0$, where $\Omega'_{i}$ is an open neighborhood of $\overline{\Omega}_{i}$ and is disjoint with $\Omega'_{j}$ for $j\neq i$.
Therefore, it follows from $\|\nabla \phi_{\epsilon}\|_{L^2(\cup_{i=1}^{ k}\Omega_{\epsilon,i})^c}=o_{\epsilon}(1)$ that we deduce $
|\phi_{i,\epsilon}|^2_2
=H_{\epsilon,i}(\phi_{\epsilon})+o_{\epsilon}(1)=s_i^0+o_{\epsilon}(1)$,
where  $o_{\epsilon}(1)\rightarrow0$ as $\epsilon\rightarrow 0$ uniformly. Then we have
\begin{align}
E_{b_i}(\phi_{i,\epsilon})\geq E_{b_i}(|\phi_{i,\epsilon}|^2_2)\geq E_{b_i}(s_i^0)+o_{\epsilon}(1).\label{e33}
\end{align}
From (\ref{e32}) and (\ref{e33}) we deduce that
\begin{align*}
\sigma_0-\sigma\geq& E_{\epsilon}(\phi_{\epsilon})\geq \sum\limits_{i=1}^{ k}\left(\frac{1}{2}\int_{\Omega_{\epsilon,i}}|\nabla \phi_{\epsilon}|^2-b_i
\int_{\Omega_{\epsilon,i}}\int_{\Omega_{\epsilon,i}}\frac{|\phi_{\epsilon}|^p|\phi_{\epsilon}|^p}{|x-y|^{\mu}}\right)+o_{\epsilon}(1)
\\ =&\sum\limits_{i=1}^{ k}E_{b_i}(\phi_{i,\epsilon})+o_{\epsilon}(1)
\geq \sigma_0+o_{\epsilon}(1),
\end{align*}
which is a contradiction. This completes the proof of proposition \ref{L3.4} in the $L^2$-subcritial case.

\textbf{Case 2: $L^2$-supercritical case $p\in \left(\frac{2N-\mu+2}{N},\frac{2N-\mu}{N-2}\right)$.}

We  also define the $( k-1)$-dimensional initial path by $
\widetilde{L}_{\epsilon,t}(x)=\widetilde{K}_{\epsilon}\sum_{i=1}^{ k}t_i^{N/2}(\zeta_{\epsilon} v_i)(t_i(\cdot-x_{\epsilon}^i)),\ t=(t_1,\cdots,t_{ k})\in(0,\infty)^{ k}$. Here,
$\widetilde{K}_{\epsilon}:=\left|\sum^{ k}_{i=1}t_i^{N/2}(\zeta_{\epsilon}v_i(t_i(\cdot-x_{\epsilon}^i))\right|^{-1}_{2}$ and  $v_i\in  \overline{S}_{b_i,s^0_i}$, $i=1,\cdots,k$.
Note that a constant $h>0$ exists such that $\widetilde{L}_{\epsilon,t}\in D^{\epsilon}_{d/4}$ and $\widetilde{K}_{\epsilon}\rightarrow1$ as
 $\epsilon\rightarrow 0$ uniformly for $t\in [1-h,1+h]^{ k}$.  Similar to case 1, we also define
\begin{align*}
\sigma_{\epsilon}:=\max\limits_{t\in [1-h,1+h]^{ k}}E_{\epsilon}(\widetilde{L}_{\epsilon,t})
\end{align*}
and has the following conclusions:
$\lim\limits_{\epsilon\rightarrow0}\sigma_{\epsilon}=\sigma_0$ and there exists $k\in \left(0,\min\{\sigma_1,\frac{d\omega}{16}\}\right)$ such that for any $\epsilon\in (0,\epsilon_{\sigma})$
\begin{align}\label{e34}
\sup\limits_{t\in\partial[1-h,1+h]^{ k}}E_{\epsilon}(\widetilde{L}_{\epsilon,t})<\sigma_0-2\sigma.
\end{align}
From Lemma \ref{L3.5} and (\ref{e34}), we conclude that
\begin{align*}
&\theta\circ \widetilde{L}_{\epsilon,t}=\widetilde{L}_{\epsilon,t},\ \text{if}\ t\in \partial[1-h,1+h]^{ k},\\
&E_{\epsilon}\circ\theta\circ \widetilde{L}_{\epsilon,t}\leq \sigma_0-\sigma\ \text{and}\ \theta\circ \widetilde{L}_{\epsilon,t}\in D^{\epsilon}_d,\
\text{for}\ t\in [1-h,1+h]^{ k}.
\end{align*}
From the above equalities, one has
\begin{align}\label{e35}
\begin{cases}
E_{\epsilon}\circ J_{\epsilon}\circ \theta\circ \widetilde{L}_{\epsilon,t}\leq \sigma_0-\sigma,\
\text{if}\ t\in [1-h,1+h]^{ k},\\
\|\nabla J_{\epsilon}(\theta(\widetilde{L}_{\epsilon,t}))\|_{L^2(\cup_{i=1}^{ k}\Omega_{\epsilon,i})^c}= o_{\epsilon}(1),\\
J_{\epsilon}\circ\theta\circ \widetilde{L}_{\epsilon,t}(x)=\theta\circ \widetilde{L}_{\epsilon,t}(x),\quad \text{if}\
x\in \cup_{i=1}^{ k}\Omega_{\epsilon,i},\\
J_{\epsilon}\circ\theta\circ \widetilde{L}_{\epsilon,t}(x)=\widetilde{L}_{\epsilon,t}(x),\quad \text{if}\ t\in \partial[1-h,1+h]^{ k}\ \text{and}\
x\in \cup_{i=1}^{ k}\Omega_{\epsilon,i}.
\end{cases}
\end{align}

Set
\begin{align*}
\widetilde{H}_{\epsilon,i}(v)=\left(\int_{\Omega_{\epsilon,i}}|\nabla v|^2\right)^{\frac{1}{-Np+2+2N-\mu}}\left(\frac{b_i(Np-\mu)}
{2p}\int_{\Omega_{\epsilon,i}}\int_{\Omega_{\epsilon,i}}\frac{|v|^p|v|^p}{|x-y|^{\mu}}\right)^{\frac{1}{Np-2-\mu}}, v\in H^1(\mathbb{R}^N)
\end{align*}
and  the map $\widetilde{\rho}_{\epsilon}(t)=\left(\widetilde{H}_{\epsilon,1}\circ J_{\epsilon}\circ\theta\circ \widetilde{L}_{\epsilon,t_{1}},\cdots,
\widetilde{H}_{\epsilon, k}\circ J_{\epsilon}\circ\theta\circ \widetilde{L}_{\epsilon,t_{ k}}\right)$.
Since $\mathcal{P}_{b_i}(v_i)=0$, i.e.,
\begin{align*}
\int_{\mathbb{R}^N}|\nabla v_i|^2=\frac{b_i(Np-\mu)}
{2p}\int_{\mathbb{R}^N}\int_{\mathbb{R}^N}\frac{|v_i|^p|v_i|^p}{|x-y|^{\mu}},
\end{align*} we easily get that
\begin{align*}
\lim\limits_{\epsilon\rightarrow0}\widetilde{H}_{\epsilon,i}(\widetilde{L}_{\epsilon,t})
=&\left(\int_{\mathbb{R}^N}|\nabla t_i^{N/2}v_i(t_i\cdot)|^2\right)^{\frac{1}{-Np+2+2N-\mu}}\left(\frac{b_i(Np-\mu)}
{2p}\int_{\mathbb{R}^N}\int_{\mathbb{R}^N}
\frac{|t_i^{N/2}v_i(t_i\cdot)|^p|t_i^{N/2}v_i(t_i\cdot)|^p}{|x-y|^{\mu}}\right)^{\frac{1}{Np-2-\mu}}\\=&t_i
\end{align*}
uniformly for $t\in [1-h,1+h]^{ k}$ ($i=1,\cdots, k$).
Then it follows from (\ref{e35}) that $t^0:=(1,\cdots,1)\in \mathbb{R}^{ k}\setminus \widetilde{\rho}_{\epsilon}(\partial [1-h,1+h]^{ k})$ and the Brower degree $\deg(\widetilde{\rho}_{\epsilon}, [1-h,1+h]^{ k},t^0)=\deg(id, [1-h,1+h]^{ k},t^0)=1$ for $\epsilon$ small.
This shows that $\overline{t}\in[1-h,1+h]^{ k}$ exists such that
\begin{equation*}
\widetilde{H}_{\epsilon,i}(J_{\epsilon}(\theta(\widetilde{L}_{\epsilon,\overline{t}})))=1,\quad \text{for each}\ i=1,\cdots, k.
\end{equation*}
Now we denote $v_{\epsilon}^0:=J_{\epsilon}(\theta(\widetilde{L}_{\epsilon,\overline{t}}))$ and $v_{\epsilon}^i:=\zeta_{\epsilon,i}v_{\epsilon}^0\left(\sum_{i=1}^{ k}|\zeta_{\epsilon,i}v_{\epsilon}^0|^2_2\right)^{-1},\ i=1,\cdots, k$.
It follows from $\|\nabla v_{\epsilon}^0\|_{L^2(\cup_{i=1}^{ k}\Omega_{\epsilon,i})^c}= o_{\epsilon}(1)$ that $\sum_{i=1}^{ k}|\zeta_{\epsilon,i}v_{\epsilon}^0|^2_2=1+o_{\epsilon}(1)$
and
\begin{align*}
t_{\epsilon,i}:=&\left(\int_{\mathbb{R}^N}|\nabla v_{\epsilon}^i|^2\right)^{\frac{1}{-Np+2+2N-\mu}}\left(\frac{b_i(Np-\mu)}
{2p}\int_{\mathbb{R}^N}\int_{\mathbb{R}^N}\frac{|v_{\epsilon}^i|^p|v_{\epsilon}^i|^p}{|x-y|^{\mu}}\right)^{\frac{1}{Np-2-\mu}}\\
=&\widetilde{H}_{\epsilon,i}(v_{\epsilon}^0)+o_{\epsilon}(1)=1+o_{\epsilon}(1),
\end{align*}
where $o_{\epsilon}(1)\rightarrow 0$ as $\epsilon\rightarrow 0$.
Combining with  Lemma \ref{L2-4} $(ii)$,
we have
\begin{align}\label{e36}
\sum\limits_{i=1}^{ k}E_{b_i}\left(v_{\epsilon}^i\right)+o_{\epsilon}(1)=
\sum\limits_{i=1}^{ k}E_{b_i}\left(t_{\epsilon,i}^{-\frac{N}{2}}v_{\epsilon}^i(t_{\epsilon,i}^{-1}\cdot)\right)\geq \sum_{i=1}^{ k}E_{b_i}(|v_{\epsilon}^i|^2_2)\geq \sigma_0
.
\end{align}
 By  (\ref{e35}) and (\ref{e36}) , we obtain that
 \begin{align*}
  \sigma_0-\sigma\geq E_{\epsilon}(v_{\epsilon}^0)\geq &\sum\limits_{i=1}^{ k}\left(\int_{\Omega_{\epsilon,i}}|\nabla v_{\epsilon}^0|^2
  -\frac{b_i}{2p}\int_{\Omega_{\epsilon,i}}\int_{\Omega_{\epsilon,i}}
  \frac{|v_{\epsilon}^0|^p|v_{\epsilon}^0|^p}{|x-y|^{\mu}}\right)+o_{\epsilon}(1)\\
  \geq&\sum\limits_{i=1}^{ k}E_{b_i}(v_{\epsilon}^i)+o_{\epsilon}(1)\geq \sigma_0+o_{\epsilon}(1),
 \end{align*}
which is a contraction. Hence, we complete the proof of Proposition \ref{L3.4} in the $L^2$-supercritical case.

\textbf{Proof of Theorem \ref{T1.5}:} It is a direct consequence of Lemma \ref{L2-6} and Proposition \ref{L3.4}.

\section{Appendix}

\textbf{The proof of Lemma \ref{L3.5}:} Define locally Lipschitz functions $Z_1, Z_2:  S_{1}\rightarrow [0,1]$ such that
\begin{align*}
Z_1(v)=\begin{cases} 0,\quad\ v\notin D^{\epsilon}_d,\\1,\quad\ v\in D^{\epsilon}_{d/2}, \end{cases}\quad
Z_2(v)=\begin{cases} 0,\quad\ v\notin \left[\sigma_0-2\sigma\leq E_{\epsilon}\right]\\1,\quad\ v\in \left[\sigma_0-\sigma\leq E_{\epsilon}\right],.\end{cases}
\end{align*}
Then we obtain a deformation flow defined by
\begin{align*}
\begin{cases}
\frac{\rm{d}\overline{\theta}}{\rm{d}s}(s,v)=-Z_1(\overline{\theta}(s,v))Z_2(\overline{\theta}(s,v))
\beta_{\epsilon}(\overline{\theta}(s,v)),\\
\overline{\theta}(0,v)=v,
\end{cases}
\end{align*}
has a unique solution $\overline{\theta}(\cdot,v)$ and
$\overline{\theta}\in C([0,+\infty)\times S_{1}, S_{1})$ satisfies that $\overline{\theta}(\cdot,v)\in D^{\epsilon}_d$ for any $\ v\in D^{\epsilon}_d$ and $\overline{\theta}(\cdot,v)\equiv v$ for any $v\notin \left[\sigma_0-2\sigma_1\leq E_{\epsilon}\right]\cap D^{\epsilon}_d$.
Here $\beta_{\epsilon}: S_{1}\rightarrow H^1(\mathbb{R}^N)$ is a locally Lipschitz continuous tangent vector field satisfying
$\beta_{\epsilon}(v)\in T_{v} S_{1}$ and
\begin{align}\label{e37}
&\|\beta_{\epsilon}(v)\|\leq1\ \text{and}\ \langle E'_{\epsilon}|_{ S_{1}}(v),\beta_{\epsilon}(v)\rangle\geq 0\ \text{for}\ v\in S_{1},\notag\\
&\langle E'_{\epsilon}|_{ S_{1}}(v),\beta_{\epsilon}(v)\rangle\geq\frac{\omega_{\epsilon}}{2}\ \text{for}\ v\in D^{\epsilon}_d\cap
[\sigma_0-2\sigma_1\leq E_{\epsilon}\leq \sigma_{\epsilon}],\\
&\langle E'_{\epsilon}|_{ S_{1}}(v),\beta_{\epsilon}(v)\rangle\geq \frac{\omega}{2}\ \text{for}
\ v\in (D^{\epsilon}_d\setminus D^{\epsilon}_{d/4})\cap[\sigma_0-2\sigma_1\leq E_{\epsilon}\leq \sigma_{\epsilon}].\notag
\end{align}
The existence of such a $\beta_{\epsilon}$ follows from \cite{R, W}.
For $v\in  S_{1}$ and $s\geq0$, we compute that
\begin{align}\label{e38}
\frac{\rm{d}}{\rm{d}s} E_{\epsilon}(\overline{\theta}(s,v))
\begin{cases}
=-Z_1(\overline{\theta}(s,v))Z_2(\overline{\theta}(s,v))
\langle E'_{\epsilon}|_{ S_{1}}(\overline{\theta}(s,v)),\beta_{\epsilon}(\overline{\theta}(s,v))\rangle\leq0,\\
\leq -\omega_{\epsilon}/2,\ \text{for}\ v\in D^{\epsilon}_d\cap
[\sigma_0-2\sigma_1\leq E_{\epsilon}\leq \sigma_{\epsilon}],\\
\leq -\omega/2,\ \text{for}
\ v\in (D^{\epsilon}_d\setminus D^{\epsilon}_{d/4})\cap[\sigma_0-2\sigma_1\leq E_{\epsilon}\leq \sigma_{\epsilon}].
\end{cases}
\end{align}
From (\ref{e38}), we can set
$\theta:=\overline{\theta}\left(\frac{4\sigma}{\omega_{\epsilon}},v\right)$. Then it is obvious that $\theta$  satisfies $(i)$ and $(ii)$.
To prove $(iii)$, we argue by contradiction and assume that there exists $v_{\epsilon}\in D^{\epsilon}_{d/4}\cap\left[E_{\epsilon}\leq \sigma_0\right]$ such that
\begin{equation}
E_{\epsilon}(\overline{\theta}(s,v_{\epsilon}))> \sigma_0-\sigma,\quad \text{for all}\ s\in\left[0,\frac{4\sigma}{\omega_{\epsilon}}\right].\label{e39}
\end{equation}
If $\overline{\theta}\in D^{\epsilon}_{d/2}$ for every $s\in \left[0,\frac{4\sigma}{\omega_{\epsilon}}\right]$, then using (\ref{e37}) and (\ref{e38}) one has
\begin{align*}
E_{\epsilon}(\theta(v_{\epsilon}))=E_{\epsilon}\left(\overline{\theta}\left(\frac{4\sigma}{\omega_{\epsilon}},
v_{\epsilon}\right)\right)
=&\int^{\frac{4\sigma}{\omega_{\epsilon}}}_0\frac{\rm{d}}{\rm{d}s}E_{\epsilon}(\overline{\theta}(s,v_{\epsilon}))ds
+E_{\epsilon}(v_{\epsilon})
\\ \leq&-\int_{0}^{\frac{4\sigma}{\omega_{\epsilon}}}\frac{\omega_{\epsilon}}{2}ds+\sigma_{\epsilon}\\
=&\sigma_{\epsilon}-2\sigma,
\end{align*}
which contradicts with (\ref{e39}). If $\overline{\theta}(s,v_{\epsilon})\notin D^{\epsilon}_{d/2}$ for some $s\in\left[0,\frac{4\sigma}{\omega_{\epsilon}}\right]$,
then we can find an interval $[s_1, s_2]\subset\left[0,\frac{4\sigma}{\omega_{\epsilon}}\right]$ such that $\overline{\theta}(s_1,v_{\epsilon})\in\partial D^{\epsilon}_{d/4}$,
 $\overline{\theta}(s_2,v_{\epsilon})\in\partial D^{\epsilon}_{d/2}$
and $\overline{\theta}(s,v_{\epsilon})\subset D^{\epsilon}_{d/2}\setminus D^{\epsilon}_{d/4}$ for all $s\in (s_1, s_2)$. Since
\begin{equation*}
s_2-s_1\geq\int_{s_1}^{s_2}\|\beta_{\epsilon}(\overline{\theta}(s,v_{\epsilon}))\|ds\geq \left\|
\int_{s_1}^{s_2}\frac{\rm{d}\overline{\theta}}{\rm{d}s}(s,v_{\epsilon})ds\right\|=\|\overline{\theta}(s_2,v_{\epsilon})
-\overline{\theta}(s_1,v_{\epsilon})\|\geq d/4,
\end{equation*}
it follows from (\ref{e38}) and (\ref{e39}) that
\begin{align*}
E_{\epsilon}(\theta(v_{\epsilon}))=E_{\epsilon}\left(\overline{\theta}\left(\frac{4\sigma}{\omega_{\epsilon}},v_{\epsilon}\right)\right)
\leq E_{\epsilon}(\overline{\theta}(s_2,v_{\epsilon}))\leq &\int_{s_1}^{s_2}
\frac{\rm{d}}{\rm{d}s}E_{\epsilon}(\overline{\theta}(s,v_{\epsilon}))+
E_{\epsilon}(\overline{\theta}(s_1,v_{\epsilon}))\\
\leq& -\frac{\omega}{2}(s_2-s_1)+\sigma_{\epsilon}\\
\leq& \sigma_{\epsilon}-2\sigma,
\end{align*}
which also contradicts with (\ref{e39}).

\textbf{The proof of Proposition \ref{L3.6}:}

Set
$\vartheta_{\epsilon}(u): D^{\epsilon}_d\rightarrow \mathbb{R}$ by $
\vartheta_{\epsilon}(u)=|\nabla u|^2_{L^2(\cup_{i=1}^{ k}\Omega_{\epsilon,i})^c}$,
which will play important roles to construct $J_{\epsilon}$.
\begin{lemma}\label{L4.1}
Suppose that $u\in D^{\epsilon}_d$. Then there exists a $\xi_u\in H^1_0(\cup_{i=1}^{ k}(\mathcal{A}_{\epsilon,i})^{\tau})^c$ and satisfies the following conditions:

$(i)$ There holds $
\|\xi_u\|\leq 2d$.

$(ii)$ We have
\begin{align}
E'_{\epsilon}(u)\xi_u\geq C\left(|\nabla u|^2_{L^2(\cup_{i=1}^{ k}\Omega_{\epsilon,i})^c}-c_{\epsilon}\right),\label{e40}
\end{align}
where $c_{\epsilon}$ is independent of $p, u$ and satisfies $\lim\limits_{\epsilon\rightarrow0}c_{\epsilon}=0$.

$(iii)$ For any $u\in D^{\epsilon}_d$,
\begin{align}
\vartheta'_{\epsilon}(u)\xi_u\geq C\left(|\nabla u|^2_{L^2(\cup_{i=1}^{ k}\Omega_{\epsilon,i})^c}-c_{\epsilon}\right)\label{e41}
\end{align}
where $C, c_{\epsilon}>0$ are as in $(ii)$.
\end{lemma}

\begin{proof}
$(i)$
 Set $
\xi_u(x)=\sum_{i=1}^{ k}\chi_{\epsilon}^i(x)u(x)\in H^1_0(\cup_{i=1}^{ k}(\mathcal{A}_{\epsilon,i})^{\tau})^c$
with $\chi_{\epsilon}^i(x)\in C_0^{\infty}\left(\mathbb{R}^N,[0,1]\right)$ a cut-off function such that $\chi_{\epsilon}^i=1$ on
$(\Omega_{\epsilon,i})^c$ and $\chi_{\epsilon}^i=0$ on
$(\mathcal{A}_{\epsilon,i})^{\tau}$ for each $i=1,\cdots, k$. Clearly, $\|\xi_u\|_{H^1_0(\cup_{i=1}^{ k}(\mathcal{A}_{\epsilon,i})^{\tau})^c}\leq 2d$.

$(ii)$ We compute that
\begin{align*}
E'_{\epsilon}(u)\xi_u=&\int_{\mathbb{R}^N}\nabla u(x)\nabla \sum_{i=1}^{ k}\chi_{\epsilon}^i(x)u(x)dx\\
&-\int_{\mathbb{R}^N}\int_{(\cup_{i=1}^{ k}(\mathcal{A}_{\epsilon,i})^{\tau})^c}
\frac{Q(\epsilon x)Q(\epsilon y)|u(x)|^p|u(y)|^{p}\sum_{i=1}^{ k}\chi_{\epsilon}^i(y)}{|x-y|^{\mu}}\\
\geq &\int_{(\cup_{i=1}^{ k}\Omega_{\epsilon,i})^c}|\nabla u|^2dx+\int_{\cup_{i=1}^{ k}(\Omega_{\epsilon,i}\setminus(\mathcal{A}_{\epsilon,i})^{\tau})}u\nabla u\nabla \sum_{i=1}^{ k}\chi_{\epsilon}^i(x)dx\\
&-\int_{\mathbb{R}^N}\int_{(\cup_{i=1}^{ k}(\mathcal{A}_{\epsilon,i})^{\tau})^c}
\frac{Q(\epsilon x)Q(\epsilon y)|u(x)|^p|u(y)|^{p}}{|x-y|^{\mu}}\\
=&\int_{(\cup_{i=1}^{ k}\Omega_{\epsilon,i})^c}|\nabla u|^2dx+(I)-(II).
\end{align*}
It is easy to deduce that
\begin{align*}
(I)=&\int_{\cup_{i=1}^{ k}(\Omega_{\epsilon,i}\setminus(\mathcal{A}_{\epsilon,i})^{\tau})}u\nabla u\nabla \sum_{i=1}^{ k}\chi_{\epsilon}^i(x)dx\\
\leq&C|\nabla u|^4_{ L^2(\cup_{i=1}^{ k}(\Omega_{\epsilon,i}\setminus(\mathcal{A}_{\epsilon,i})^{\tau}))}
\leq c_1^{\epsilon},
\end{align*}
where $\lim\limits_{\epsilon\rightarrow 0}c_1^{\epsilon}=0$.
From Propositions \ref{L2-1} and  \ref{L2-2}, we calculate that
\begin{align*}
(II)=&\int_{\mathbb{R}^N}\int_{(\cup_{i=1}^{ k}(\mathcal{A}_{\epsilon,i})^{\tau})^c}
\frac{Q(\epsilon x)Q(\epsilon y)|u(x)|^p|u(y)|^{p}}{|x-y|^{\mu}}
\\ =&\int_{\mathbb{R}^N}\int_{(\cup_{i=1}^{ k}\Omega_{\epsilon,i})^c}
\frac{Q(\epsilon x)Q(\epsilon y)|u(x)|^p|u(y)|^{p}}{|x-y|^{\mu}}+
\int_{\mathbb{R}^N}\int_{\cup_{i=1}^{ k}\Omega_{\epsilon,i}\setminus(\mathcal{A}_{\epsilon,i})^{\tau}}
\frac{Q(\epsilon x)Q(\epsilon y)|u(x)|^p|u(y)|^{p}}{|x-y|^{\mu}},\\
\leq&\int_{(\cup_{i=1}^{ k}\Omega_{\epsilon,i})^c}\int_{(\cup_{i=1}^{ k}\Omega_{\epsilon,i})^c}
\frac{Q(\epsilon x)Q(\epsilon y)|u(x)|^p|u(y)|^{p}}{|x-y|^{\mu}}+\int_{\cup_{i=1}^{ k}(\mathcal{A}_{\epsilon,i})^{\tau}}\int_{(\cup_{i=1}^{ k}\Omega_{\epsilon,i})^c}
\frac{Q(\epsilon x)Q(\epsilon y)|u(x)|^p|u(y)|^{p}}{|x-y|^{\mu}}\\
&+
2\int_{\mathbb{R}^N}\int_{\cup_{i=1}^{ k}\Omega_{\epsilon,i}\setminus(\mathcal{A}_{\epsilon,i})^{\tau}}
\frac{Q(\epsilon x)Q(\epsilon y)|u(x)|^p|u(y)|^{p}}{|x-y|^{\mu}}\\
\leq & c_2^{\epsilon}+C\|u\|^{2p}_{H^1_0(\cup_{i=1}^{ k}\Omega_{\epsilon,i})^c}.
\end{align*}
where $\lim\limits_{\epsilon\rightarrow 0}c_2^{\epsilon}=0$.
Hence, it holds $
E'_{\epsilon}(u)\xi_u\geq \int_{(\cup_{i=1}^{ k}\Omega_{\epsilon,i})^c}|\nabla u|^2dx-c_1^{\epsilon}-c_2^{\epsilon}-C\|u\|^{2p}_{H^1_0(\cup_{i=1}^{ k}\Omega_{\epsilon,i})^c}$.
Controlling $d$ sufficiently small if necessary, we may assume that for some $C>0$,
\begin{align*}
 \int_{(\cup_{i=1}^{ k}\Omega_{\epsilon,i})^c}|\nabla u|^2dx-C\|u\|^{2p}_{H^1_0(\cup_{i=1}^{ k}\Omega_{\epsilon,i})^c}\geq C\int_{(\cup_{i=1}^{ k}\Omega_{\epsilon,i})^c}|\nabla u|^2dx.
\end{align*}
Setting $c_{\epsilon}=\left(c_1^{\epsilon}-c_2^{\epsilon}\right)/C$, we obtain (\ref{e40}).

$(iii)$ (\ref{e41}) can be shown as in $(ii)$.
\end{proof}

From the conclusions of Proposition \ref{L4.1}, we can obtain the following Proposition.

\begin{proposition}\label{L4.2}
Let $c_{\epsilon}>0$ be as in Proposition $\ref{L4.1}$ $(ii)$. Then
there is a local Lipschitz vector field $
\gamma(u):D^{\epsilon}_d\rightarrow H^1_0(\cup_{i=1}^{ k}(\mathcal{A}_{\epsilon,i})^{\tau})^c$
such that, for all $u\in D^{\epsilon}_d$, there holds $\|\gamma(u)\|_{H^1}\leq 1$ and
\begin{align*}
E'_{\epsilon}(u)\gamma(u),
\vartheta'_{\epsilon}(u)\gamma(u)\geq C\left(|\nabla u|^2_{L^2(\cup_{i=1}^{ k}\Omega_{\epsilon,i})^c}-c_{\epsilon}\right),\ \text{if}\ \vartheta_{\epsilon}(u)>c_{\epsilon}.
\end{align*}
\end{proposition}
\begin{proof}
Let $\gamma(u)=\frac{\xi_u}{\|\xi_u\|}$ and $u\in D_d^{\epsilon}$. Then by Lemma \ref{L4.1}, the proof is completed.
\end{proof}

We consider the following ODE in $H$: for $u\in D^{\epsilon}_d$,
\begin{align}\label{e42}
\begin{cases}
\frac{dZ}{dT}=-\Gamma(\vartheta_{\epsilon}(Z))\gamma(Z),\\
Z(0,u)=u,
\end{cases}
\end{align}
where the function $\Gamma(r)\in C^{\infty}([0,\infty),[0,1])$ satisfies that $\Gamma(r)=1$ if $r\geq2c_{\epsilon}$ and $\Gamma(r)=0$ if $ r\leq c_{\epsilon}$.
Clearly, problem (\ref{e42}) has a unique solution $Z_T(u)$. Moreover, $Z_T(u)$ has the following properties.

\begin{lemma}\label{L4.3}
Let $Z_T(u)$ be the unique solutions of (\ref{e42}). Then there holds:
\begin{itemize}
\item [$(i)$] $Z_T(u)\in F_d^{\epsilon}$ for all $T>0$ and $u\in D^{\epsilon}_d$.

\item [$(ii)$]It holds
\begin{align*}
\frac{d}{dT}\vartheta_{\epsilon}(Z_T(u))
\begin{cases}
\leq 0,\ \text{for all}\ u\in D^{\epsilon}_d,\\
<0,\ \text{if}\ \vartheta_{\epsilon}(Z_T(u))>c_{\epsilon}.
\end{cases}
\end{align*}

\item [$(iii)$]
Let \begin{align*}
\widetilde{Z}_T(u)=
\begin{cases}
u(x)\ &\text{for}\ x\in \cup_{i=1}^{ k}(\mathcal{A}_{\epsilon,i})^{\tau},\\
Z_T(u)\ &\text{for}\ x\in \left(\cup_{i=1}^{ k}(\mathcal{A}_{\epsilon,i})^{\tau}\right)^c.
\end{cases}
\end{align*}
Then it holds $\frac{d}{dT}E_{\epsilon}(\widetilde{Z}_T(u))\leq0$.

\item [$(iv)$] There exists $T_{\epsilon}>0$ such that for all $u\in D^{\epsilon}_d$, $\vartheta_{\epsilon}(Z_{T_{\epsilon},u})\leq c_{\epsilon}$.
\end{itemize}
\end{lemma}
\begin{proof}
 First we show $(ii)$. Since $\vartheta_{\epsilon}(Z_T(u))>c_{\epsilon}$ implies $\Gamma(\vartheta_{\epsilon}(Z_T(u)))>0$, it follows from Proposition \ref{L4.2} that
\begin{align*}
\frac{d}{dT}\vartheta_{\epsilon}(Z_T(u))=-\Gamma(\vartheta_{\epsilon}(Z_T(u)))\vartheta'_{\epsilon}(Z_T(u))\gamma(Z_T(u))<0.
\end{align*}
Thus,  $(ii)$ holds.

By $(ii)$, we get that $\vartheta_{\epsilon}(Z_T(u))$ is non-increasing. This indicates  $Z_T(u)$ exists and satisfies $Z_T(u)\in D^{\epsilon}_d$ for $T\in [0,\infty)$. Hence
$(i)$ follows.
$(iv)$ follows from $(ii)$ easily.
 Noting $Z_T(u)|_{x\in \cup_{i=1}^{ k}\partial(\mathcal{A}_{\epsilon,i})^{\tau}}$ is independent of $T$, we have
$\widetilde{Z}_T(u)\in H^1_0(\mathbb{R}^N)$. Therefore $(iii)$ follows from Proposition \ref{L4.2} $(ii)$.
\end{proof}

\textbf{Proof of Proposition $\ref{L3.6}$:} For $u\in D^{\epsilon}_d$, we define $
J_{\epsilon}(u)=\widetilde{Z}_{T_{\epsilon}}(u)$.
The desired properties $(\ref{e28})-(\ref{e30})$ follows from Lemma \ref{L4.3}.

\section*{Acknowledgments}
H. Zhang is supported by Hunan Provincial Innovation Foundation For Postgraduate (No. CX20230219) and the Fundamental Research Funds for the Central Universities of Central South University (No. 1053320220488). H. Chen is supported by the National Natural Science Foundation of China (Grant No. 12071486).

\end{document}